\documentclass[10pt,oneside,english,reqno]{amsart}
\usepackage{lmodern}
\usepackage[T1]{fontenc}
\usepackage[latin9]{inputenc}
\usepackage{babel}
\usepackage{amsthm}
\usepackage{amssymb}
\usepackage{graphicx}
\usepackage[unicode=true,pdfusetitle,
 bookmarks=true,bookmarksnumbered=false,bookmarksopen=false,
 breaklinks=false,pdfborder={0 0 1},backref=false,colorlinks=false]
 {hyperref}

\makeatletter

\providecommand{\tabularnewline}{\\}

\theoremstyle{plain}
\newtheorem{thm}{\protect\theoremname}[section]
  \theoremstyle{definition}
  \newtheorem{defn}[thm]{\protect\definitionname}
  \theoremstyle{definition}
  \newtheorem{example}[thm]{\protect\examplename}
  \theoremstyle{plain}
  \newtheorem{assumption}[thm]{\protect\assumptionname}
  \theoremstyle{plain}
  \newtheorem{lem}[thm]{\protect\lemmaname}
  \theoremstyle{remark}
  \newtheorem{rem}[thm]{\protect\remarkname}
  \theoremstyle{plain}
  \newtheorem{prop}[thm]{\protect\propositionname}
  \theoremstyle{plain}
  \newtheorem{cor}[thm]{\protect\corollaryname}

\newcommand{\lip}{\mbox{\rm lip}}
\DeclareMathOperator*{\slimsup}{LIMSUP}

\newcommand{\cls}{\mbox{\rm{cl-sets}}_{\neq \emptyset}}
\newcommand{\gph}{\mbox{\rm gph}}
\newcommand{\cl}{\mbox{\rm cl}}
\newcommand{\dom}{\mbox{\rm dom}}
\newcommand{\co}{\mbox{\rm co}}

\numberwithin{equation}{section}
\numberwithin{figure}{section}
\title[Characterizing generalized derivatives of set-valued maps]{Characterizing generalized derivatives of set-valued maps: Extending the tangential and normal approaches}

\makeatother

  \providecommand{\assumptionname}{Assumption}
  \providecommand{\corollaryname}{Corollary}
  \providecommand{\definitionname}{Definition}
  \providecommand{\examplename}{Example}
  \providecommand{\lemmaname}{Lemma}
  \providecommand{\propositionname}{Proposition}
  \providecommand{\remarkname}{Remark}
\providecommand{\theoremname}{Theorem}

\begin{document}

\author{C.H. Jeffrey Pang}

\curraddr{Department of Mathematics, National University of Singapore, Block
S17 05-10, 10 Lower Kent Ridge Road, Singapore 119076 }

\email{matpchj@nus.edu.sg}


\date{\today}
\begin{abstract}
For a set-valued map, we characterize, in terms of its (unconvexified
or convexified) graphical derivatives near the point of interest,
positively homogeneous maps that are generalized derivatives in the
sense of \cite{set_diff}. This result generalizes the Aubin criterion
in \cite{DQZ06}. A second characterization of these generalized derivatives
is easier to check in practice, especially in the finite dimensional
case. Finally, the third characterization in terms of limiting normal
cones and coderivatives generalizes the Mordukhovich criterion in
the finite dimensional case. The convexified coderivative has a bijective
relationship with the set of possible generalized derivatives. We
conclude by illustrating a few applications of our result.

\tableofcontents{}
\end{abstract}
\maketitle

\section{Introduction}

We say that $S$ is a \emph{set-valued map }or a \emph{multifunction},
denoted by $S:X\rightrightarrows Y$, if $S(x)\subset Y$ for all
$x\in X$. There are many examples of set-valued maps in optimization
and related areas. For example, the generalized derivatives of a nonsmooth
function, the feasible set of a parametric optimization problem, and
the set of optimizers to a parametric optimization problem may be
profitably viewed as set-valued maps. 

The Lipschitz analysis of a set-valued map (more precisely, the Aubin
property) is equivalent to the metric regularity of its inverse. Metric
regularity is in turn used to derive stability conditions for nonsmooth
problems, and thus identify when a problem is ill conditioned and
cannot be easily resolved by any numerical method. One can identify
metric regularity from graphical derivatives or coderivatives. We
will discuss these criteria later in more detail.

We now illustrate how the Lipschitz analysis of set-valued maps can
be helpful in the analysis of optimization problems. Consider the
problem $P(u,v)$ defined by 
\begin{equation}
P(u,v):\inf_{x\in S(u)}v^{T}x,\label{eq:first-example}
\end{equation}
where $S:U\rightrightarrows X$ is a set-valued map. A profitable
way of analyzing $P(u,v)$ is by studying the set-valued map $S$.
The Lipschitz continuity of $P(\cdot,\bar{v})$ is established if
$S$ is Lipschitz in the Pompieu-Hausdorff distance, which can be
easily checked when $S$ has a closed convex graph through the Robinson-Ursescu
Theorem. See for example \cite{Cla83,DR09}.

It is natural to ask whether a first order analysis of set-valued
maps can be a more effective tool in the analysis of optimization
and equilibrium problems than a Lipschitz analysis, but we need to
first build the basic tools. This paper studies how a first order
analysis of a set-valued map may be obtained from the tangent and
normal cones of its graph, generalizing the Aubin and Mordukhovich
criteria. We will apply our results to study the set-valued map of
feasible points satisfying a set of equalities and inequalities in
Proposition \ref{pro:RW-E9.44}.

The Aubin criterion as presented in \cite{DQZ06} characterizes the
Lipschitz properties of a set-valued map $S:X\rightrightarrows Y$
using the tangent cones of its graph $\gph(S)$. Here, the \emph{graph
}$\gph(S)$ is the set $\{(x,y)\mid y\in S(x)\}\subset X\times Y$.
One contribution of this paper is to characterize the generalized
derivatives, introduced in \cite{set_diff}, of the set-valued map
$S$ in terms of the tangent cones of its graph. It is usually easier
to obtain information on the tangent cones of $\gph(S)$ rather than
the generalized derivatives, so our result will play the role the
Mordukhovich criterion and the Aubin criterion currently have in Lipschitz
analysis. We now recall some standard definitions necessary to proceed.
The closed ball with center $\bar{x}$ and radius $\epsilon$ is denoted
by $\mathbb{B}_{\epsilon}(\bar{x})$, and $\mathbb{B}_{1}(0)$ is
written simply as $\mathbb{B}$.
\begin{defn}
(Positive homogeneity) Let $X$ and $Y$ be linear spaces. A set-valued
map $H:X\rightrightarrows Y$ is \emph{positively homogeneous} if
\[
0\in H(\mathbf{0})\mbox{, and }H(kw)=kH(w)\mbox{ for all }k>0\mbox{ and }w\in X.
\]

\end{defn}
A positively homogeneous map is also called a \emph{process}. The
positively homogeneous map $(H+\delta):X\rightrightarrows Y$, where
$H:X\rightrightarrows Y$ is positively homogeneous and $\delta>0$
is a real number, is defined by
\[
(H+\delta)(w):=H(w)+\delta\|w\|\mathbb{B}.
\]
Here is the definition of generalized differentiability of set-valued
maps introduced in \cite{set_diff}.
\begin{defn}
\label{def:T-diff}\cite{set_diff} (Generalized differentiability)
Let $X$ and $Y$ be normed linear spaces. Let $S:X\rightrightarrows Y$
be such that $(\bar{x},\bar{y})\in\gph(S)$, and let $H:X\rightrightarrows Y$
be positively homogeneous. The map $S$ is \emph{pseudo strictly $H$-differentiable
at $(\bar{x},\bar{y})$ }if for any $\delta>0$, there are neighborhoods
$U_{\delta}$ of $\bar{x}$ and $V_{\delta}$ of $\bar{y}$ such that
\[
S(x)\cap V_{\delta}\subset S(x^{\prime})+(H+\delta)(x-x^{\prime})\mbox{ for all }x,x^{\prime}\in U_{\delta}.
\]
If $S$ is pseudo strictly $H$-differentiable for some $H$ defined
by $H(w)=\kappa\|w\|\mathbb{B}$, where $\kappa\geq0$, then $S$
satisfies the \emph{Aubin property}, also referred to as the \emph{pseudo-Lipschitz
}property. The \emph{Lipschitz modulus} (or \emph{graphical modulus})
is the infimum of all such $\kappa$, and is denoted by $\lip\, S(\bar{x}\mid\bar{y})$.
\end{defn}
We had used $T:X\rightrightarrows Y$ to denote the positively homogeneous
map in \cite{set_diff}, but we now use $H$ to denote the positively
homogeneous map instead. We reserve $T$ to denote the tangent cone,
defined as follows.
\begin{defn}
(Tangent cones) Let $X$ be a normed linear space. A vector $w\in X$
is \emph{tangent }to a set $C\subset X$ at a point $\bar{x}\in C$,
written $w\in T_{C}(\bar{x})$, if 
\[
\frac{x_{i}-\bar{x}}{\tau_{i}}\to w\mbox{ for some }x_{i}\to\bar{x},\, x_{i}\in C,\,\tau_{i}\searrow0.
\]
The set $T_{C}(\bar{x})$ is referred to as the \emph{tangent cone
}(also called the \emph{contingent cone}) to $C$ at $\bar{x}$. 
\end{defn}
We say that $S:X\rightrightarrows Y$ is \emph{locally closed }at
$(\bar{x},\bar{y})\in\gph(S)$ if $\gph(S)\cap\mathbb{B}_{\epsilon}\big((\bar{x},\bar{y})\big)$
is a closed set for some $\epsilon>0$. 
\begin{defn}
(Graphical derivative) Let $X$ and $Y$ be normed linear spaces.
For a set-valued map $S:X\rightrightarrows Y$ such that  $(\bar{x},\bar{y})\in\gph(S)$,
the \emph{graphical derivative}, also known as the \emph{contingent
derivative}, is denoted by $DS(\bar{x}\mid\bar{y}):X\rightrightarrows Y$
and defined by
\[
\gph\big(DS(\bar{x}\mid\bar{y})\big)=T_{\scriptsize\gph(S)}(\bar{x},\bar{y}).
\]
The \emph{convexified graphical derivative }is denoted by $D^{\star\star}S(x\mid y):X\rightrightarrows Y$
and is defined by 
\[
\gph\big(D^{\star\star}S(x\mid y)\big)=\cl\,\co\, T_{\scriptsize\gph(S)}(x,y),
\]
i.e., the closed convex hull of $T_{\scriptsize\gph(S)}(x,y)$. 
\end{defn}
The study of the relationship between the graphical derivative and
the graphical modulus can be traced back to the papers of Aubin and
his co-authors \cite{Aub81,AF87}, \cite[Theorem 7.5.4]{AE84} and
\cite[Theorem 5.4.3]{AF90}. The main result in \cite{DQZ06} characterizes
$\lip\, S(\bar{x}\mid\bar{y})$ in terms of the graphical derivatives,
and was named the Aubin criterion to recognize the efforts of Aubin
and his coauthors. See also \cite{Fran_Quin10}. Their result will
be stated as Theorem \ref{thm:DQZ-Aubin}. Their proof was motivated
by the proof of \cite[Theorem 3.2.4]{Aub91} due to Frankowska. Another
paper of interest on the Aubin criterion is \cite{Aub06}, where a
proof of part of the result in \cite{DQZ06} was obtained using viability
theory. 

In Asplund spaces, a different characterization of $\lip\, S(\bar{x}\mid\bar{y})$
can be obtained in terms of (limiting) coderivatives. Coderivatives
are defined in terms of the (limiting) normals of $\gph(S)$ at $(\bar{x},\bar{y})$,
so this approach can be considered as the dual approach to the Aubin
criterion. This characterization known as the Mordukhovich criterion
in \cite{RW98}. We refer to \cite{Aub06,Mor93,Mor06,RW98} for more
on the history of this result, where the contributions of Ioffe are
also highlighted. The Mordukhovich criterion has been frequently applied
to analyze many problems in nonsmooth optimization, feasibility and
equilibria. Quoting \cite{DQZ06}, we note that when $X$ is any Banach
space and $Y$ is finite dimensional, a necessary and sufficient condition
for $S:X\rightrightarrows Y$ to have the Aubin property is given
in terms of the Ioffe approximate coderivative in \cite{JT99}. We
also show how our result generalizes the Mordukhovich criterion in
the finite dimensional case, and that the convexified limiting coderivative
has a bijective relationship with the set of possible generalized
derivatives. 

The original context of the Aubin criterion was metric regularity,
while the original context of the Mordukhovich criterion was linear
openness. Metric regularity gives a description of solutions sets
to nonsmooth problems, which is one of the themes of the recent books
\cite{DR09,KK02}. Further references of metric regularity and linear
openness are \cite{Iof00,Mor06,RW98}. Other related works include
\cite{Aze06,IofSch96,Yen_Yao_Kien08}. The equivalence between the
Aubin property, metric regularity and linear openness is well known.
For readers interested in applying the results in this paper in the
context of metric regularity or linear openness, we refer to \cite[Section 7]{set_diff},
where a similar equivalence for generalized differentiability of set-valued
maps, generalized metric regularity, and generalized linear openness
is obtained.

\subsection{Contributions of this paper}

We present three sets of theorems to characterize the generalized
derivatives of a set-valued map using the tangent and normal cones.
The first set of theorems are presented in Section \ref{sec:first-char}.
Theorem \ref{thm:Aubin-crit} extends the Aubin criterion in the sense
of generalized derivatives in Definition \ref{def:T-diff}, and has
a simple proof. 

We present a second characterization of the generalized derivatives
in Section \ref{sec:second} that is easier to check in practice.
The proofs of this set of theorems depend on the first characterization
in Section \ref{sec:first-char}. More specifically, consider $S:X\rightrightarrows Y$
locally closed at $(\bar{x},\bar{y})\in\gph(S)$ and a positively
homogeneous map $H:X\rightrightarrows Y$. Consider also $\{G_{i}\}_{i\in I}$,
where $G_{i}:X\rightrightarrows Y$ are positively homogeneous and
$I$ is some index set. We impose further conditions so that $S$
is pseudo strictly $H$-differentiable at $(\bar{x},\bar{y})$ if
and only if 
\[
G_{i}(p)\cap[-H(-p)]\neq\emptyset\mbox{ for all }i\in I\mbox{ and }p\in X\backslash\{0\}.
\]
These conditions are easier to check in the finite dimensional case,
and the Clarke regular case leads to further simplifications.

Finally, a third characterization is expressed in terms of the limiting
normal cones for the finite dimensional case in Theorem \ref{thm:boris-crit},
generalizing the Mordukhovich criterion. This characterization depends
on the second characterization in Section \ref{sec:second}. The convexified
coderivative will be shown to have a bijective relationship with the
set of possible generalized derivatives in Theorem \ref{thm:co-coderv-best}.

We apply the results above in Proposition \ref{pro:RW-E9.44} to study
the generalized differentiability properties of constraint systems,
to study generalized metric regularity and linear openness, and to
estimate the convexified limiting coderivative of a set-valued map
defined as a limit of set-valued maps.

\subsection{Preliminaries and notation}

We recall other definitions in set-valued analysis needed for the
rest of this paper. We say that $S$ is \emph{closed-valued }if $S(x)$
is closed for all $x\in X$, and the definitions for compact-valuedness
and convex-valuedness are similar. For set-valued maps $S_{1}:X\rightrightarrows Y$
and $S_{2}:X\rightrightarrows Y$, we use $S_{1}\subset S_{2}$ to
denote $S_{1}(x)\subset S_{2}(x)$ for all $x\in X$, which also corresponds
to $\gph(S_{1})\subset\gph(S_{2})$.

The outer and inner norms of positively homogeneous maps will be needed
later.
\begin{defn}
(Outer and inner norms) The \emph{outer norm }$\|H\|^{+}$ and \emph{inner
norm }$\|H\|^{-}$ of a positively homogeneous map $H:X\rightrightarrows Y$
are defined by 
\begin{eqnarray*}
\|H\|^{+} & := & \sup_{\|w\|\leq1}\sup_{z\in H(w)}\|z\|\\
\mbox{and }\|H\|^{-} & := & \sup_{\|w\|\leq1}\inf_{z\in H(w)}\|z\|.
\end{eqnarray*}

\end{defn}
The positively homogeneous maps defined as fans and prefans in \cite{Iof81}
will be used frequently in the rest of this paper, and we recall their
definitions below.
\begin{defn}
\cite{Iof81} (Fans and prefans) We say that $H:\mathbb{R}^{n}\rightrightarrows\mathbb{R}^{m}$
is a \emph{prefan} if 
\begin{enumerate}
\item $H(p)$ is nonempty, convex and compact for all $p\in\mathbb{R}^{n}$. 
\item $H$ is positively homogeneous, and
\item $\|H\|^{+}$ is finite.
\end{enumerate}
In particular, $H(0)=\{0\}$. If in addition, $H(p_{1}+p_{2})\subset H(p_{1})+H(p_{2})$
for all $p_{1},p_{2}\in\mathbb{R}^{n}$, then we say that $H$ is
a \emph{fan}.
\end{defn}
It seems that prefans are not as commonly used as fans. But our characterizations
in Sections \ref{sec:second} and \ref{sec:third-char} are stated
using prefans, and we shall not use fans in this paper. Example \ref{exa:prefan}
below may help understand why prefans are more suitable.
\begin{example}
\label{exa:prefan}(Prefans over fans) This example shows a prefan
$H:\mathbb{R}\rightrightarrows\mathbb{R}$ that is not a fan such
that $S$ is pseudo strictly $H$-differentiable at $(0,0)$. Consider
$S:\mathbb{R}\rightrightarrows\mathbb{R}$ and $H:\mathbb{R}\rightrightarrows\mathbb{R}$
defined by 
\[
S(x):=(-\infty,x]\qquad\mbox{and }H(x):=\max\{0,x\}.
\]
The set-valued map $S$ is pseudo strictly $H$-differentiable at
$(0,0)$, as can be easily checked from definitions or by applying
Corollary \ref{cor:regular-fd-case} later. 
\end{example}
We recall one possible definition of inner and outer semicontinuity.
Note that the definition of inner semicontinuity may not be standard
when $X$ and $Y$ are infinite dimensional. For example, the definition
here already assumes that $X$ and $Y$ are metrizable, while the
definition in \cite{AF90} does not assume metrizability. This definition
of inner semicontinuity will be used in Definition \ref{def:gen-isc}.
\begin{defn}
(Inner and outer semicontinuity) For a closed-valued mapping $S:C\rightrightarrows Y$
and a point $\bar{x}\in C\subset X$, $S$ is \emph{inner semicontinuous
}(written \emph{isc}) \emph{with respect to $C$ at $\bar{x}$} if
for every $\rho>0$ and $\epsilon>0$, there exists a neighborhood
$V$ of $\bar{x}$ such that 
\[
S(\bar{x})\cap\rho\mathbb{B}\subset S(x)+\epsilon\mathbb{B}\mbox{ for all }x\in C\cap V.
\]
We say that $S$ is \emph{outer semicontinuous} (written \emph{osc})
\emph{with respect to $C$ at $\bar{x}$} if for every $\rho>0$ and
$\epsilon>0$, there exists a neighborhood $V$ of $\bar{x}$ such
that 
\[
S(x)\cap\rho\mathbb{B}\subset S(\bar{x})+\epsilon\mathbb{B}\mbox{ for all }x\in C\cap V.
\]

\end{defn}
Following the notation in \cite{RW98}, we say that a closed set $C\subset\mathbb{R}^{n}$
is \emph{Clarke regular }at $\bar{x}\in C$ if the tangent map $T_{C}:C\rightrightarrows\mathbb{R}^{n}$
is inner semicontinuous at $\bar{x}$. (This is equivalent to the
usual definition of Clarke regularity of a set in a finite dimensional
space through \cite[Theorem 6.26 and Corollary 6.29(b)]{RW98}.) We
shall only look at Clarke regularity of sets in finite dimensions,
in part because our definition of inner semicontinuity is nonstandard
in infinite dimensions. We say that $S$ is \emph{graphically regular}
at $(\bar{x},\bar{y})\in\gph(S)$ if $\gph(S)$ is Clarke regular
at $(\bar{x},\bar{y})$. 

We recall the definition of the outer limit of sets. For $\{x_{i}\}_{i=1}^{\infty}\subset X$,
$\bar{x}\in X$ and $C\subset X$, the notation $x_{i}\xrightarrow[C]{}\bar{x}$
means $x_{i}\in C$ for all $i$ and $x_{i}\to\bar{x}$.
\begin{defn}
(Outer limits) Let $C\subset X$. For a set-valued map $S:C\rightrightarrows Y$,
the \emph{outer limit} of $S$ at $\bar{x}\in C$, is defined by 
\[
\limsup_{x\xrightarrow[C]{}\bar{x}}S(x):=\{y\mid\mbox{there exists }x_{i}\xrightarrow[C]{}\bar{x},\, y_{i}\in S(x_{i})\mbox{ s.t. }y_{i}\to y\}.
\]

\end{defn}
Lastly, for $K\subset X$, the \emph{negative polar cone} of $K$
is denoted by $K^{0}$, and is defined by $K^{0}=\{v\mid\left\langle v,x\right\rangle \leq0\mbox{ for all }x\in K\}$.

\section{\label{sec:first-char}A first characterization: Extending the tangential
approach}

The main result of this section is Theorem \ref{thm:Aubin-crit},
where we generalize the Aubin criterion. We also mention that Lemma
\ref{lem:creeping} will be used for much of the paper later.

We list assumptions that will be used often in the rest of the paper.
\begin{assumption}
\label{ass:common}Let $X$ and $Y$ be normed linear spaces, and
assume further that $Y$ is complete (i.e, $Y$ is a Banach space).
Let $S:X\rightrightarrows Y$ be locally closed at $(\bar{x},\bar{y})\in\gph(S)$,
and let $H:X\rightrightarrows Y$ be positively homogeneous.
\end{assumption}
The following is our first characterization of the generalized derivatives
of $S$.
\begin{thm}
\label{thm:Aubin-crit}(Generalized Aubin criterion) Suppose Assumption
\ref{ass:common} holds. Consider the statements:
\begin{enumerate}
\item [(1)]$S$ is pseudo strictly $H$-differentiable at $(\bar{x},\bar{y})$.
\item [(2)]For all $\delta>0$, there are neighborhoods $U$ of $\bar{x}$
and $V$ of $\bar{y}$ such that 
\begin{flalign}
DS(x\mid y)(p)\cap[-(H+\delta)(-p)]\neq\emptyset\label{eq:aubin-crit}\\
\mbox{ for all }(x,y)\in\gph(S) & \cap[U\times V]\mbox{ and }p\in X\backslash\{0\}.\nonumber 
\end{flalign}

\item [(2$^{\prime}$)]For all $\delta>0$, there are neighborhoods $U$
of $\bar{x}$ and $V$ of $\bar{y}$ such that 
\begin{flalign}
D^{\star\star}S(x\mid y)(p)\cap[-(H+\delta)(-p)]\neq\emptyset\label{eq:aubin-convex-crit}\\
\mbox{ for all }(x,y)\in\gph(S) & \cap[U\times V]\mbox{ and }p\in X\backslash\{0\}.\nonumber 
\end{flalign}

\end{enumerate}
We then have the following:
\begin{enumerate}
\item [(a)]If $Y$ is finite dimensional and $H$ is compact-valued, then
$(1)$ implies $(2)$. 
\item [(b)]If $\|H\|^{+}<\infty$, $H$ is convex-valued and $X$ is complete,
then $(2)$ implies $(1)$. 
\item [(c)]If both $X$ and $Y$ are finite dimensional Euclidean spaces
and $H$ is a prefan, then $(1)$, $(2)$ and $(2^{\prime})$ are
equivalent.
\end{enumerate}
\end{thm}
We begin with the proof of Theorem \ref{thm:Aubin-crit}(a), which
is the simplest.
\begin{proof}
\textbf{{[}Theorem \ref{thm:Aubin-crit}(a){]}} Suppose $S$ is pseudo
strictly $H$-differentiable at $(\bar{x},\bar{y})$. Then for any
$\delta>0$, there are neighborhoods $U$ of $\bar{x}$ and $V$ of
$\bar{y}$ such that if $(x,y)\in[U\times V]\cap\gph(S)$, $p\in X\backslash\{0\}$
and $t$ is small enough so that $x+tp\in U$, then 
\begin{equation}
S(x)\cap V\subset S(x+tp)+(H+\delta)(-tp).\label{eq:KK-Lip-lsc}
\end{equation}
(This can be seen as a lower generalized differentiation property,
which resembles the \emph{lower Lipschitz} or \emph{Lipschitz lower
semicontinuous} property in \cite{KK02,KK06}.) Since $y$ lies in
the LHS of \eqref{eq:KK-Lip-lsc}, there exists some $y(t)\in S(x+tp)$
such that $y\in y(t)+(H+\delta)(-tp)$. Then 
\[
\frac{y(t)-y}{t}\in-(H+\delta)(-p).
\]
Let $\hat{y}$ be a cluster point of $\{\frac{y(t)-y}{t}\}$ as $t\searrow0$,
which exists since $Y$ is finite dimensional and $-(H+\delta)(-p)$
is compact. We have $\hat{y}\in DS(x\mid y)(p)$, so $DS(x\mid y)(p)\cap[-(H+\delta)(-p)]\neq\emptyset$
as needed. 
\end{proof}
To prove Theorem \ref{thm:Aubin-crit}(b), we need the following lemma.
\begin{lem}
\label{lem:creeping} (Estimates of generalized differentiability
from tangent cones) Suppose Assumption \ref{ass:common} holds, $X$
is complete, and assume further that $H$ is convex-valued and $\|H\|^{+}<\infty$.
Let $\delta>0$. Suppose there are neighborhoods $U$ of $\bar{x}\in X$
and $V$ of $\bar{y}\in Y$ such that whenever $(x,y)\in[U\times V]\cap\gph(S)$
and $p\in X\backslash\{0\}$, there are $(p^{\prime},q^{\prime})$
such that 
\begin{equation}
(p^{\prime},q^{\prime})\in T_{\scriptsize\gph(S)}(x,y),\mbox{ }\|p-p^{\prime}\|<\delta\|p\|\mbox{ and }q^{\prime}\in-(H+\delta)(-p).\label{eq:Aubin-compare}
\end{equation}
 Then provided $\epsilon>0$ is such that $\epsilon+\delta<1$, there
are neighborhoods $U_{\epsilon}$ of $\bar{x}$ and $V_{\epsilon}$
of $\bar{y}$ such that $x^{*}$, $x^{\circ}\in U_{\epsilon}$ implies
\[
S(x^{*})\cap V_{\epsilon}\subset S(x^{\circ})+\left(H+\delta+\epsilon+[\|H\|^{+}+\delta+\epsilon]\frac{\delta+\epsilon}{1-\delta-\epsilon}\right)(x^{*}-x^{\circ}).
\]
\end{lem}
\begin{proof}
Let $U_{\epsilon}$ and $V_{\epsilon}$ be neighborhoods of $\bar{x}$
and $\bar{y}$ respectively such that \begin{subequations}
\begin{eqnarray}
[x^{*},x^{\circ}]+(\delta+\epsilon)\|x^{*}-x^{\circ}\|\mathbb{B} & \subset & U,\label{eq:3-hypo-1}\\
x^{\circ}+[(\delta+\epsilon)+(\delta+\epsilon)^{2}]\|x^{*}-x^{\circ}\|\mathbb{B} & \subset & U,\label{eq:3-hypo-2}\\
\mbox{and }y^{*}+\frac{1}{1-\delta-\epsilon}\|x^{*}-x^{\circ}\|[\|H\|^{+}+\delta+\epsilon]\mathbb{B} & \subset & V\label{eq:3-hypo-3}
\end{eqnarray}
\end{subequations}for all $(x^{*},y^{*})\in[U_{\epsilon}\times V_{\epsilon}]\cap\gph(S)$
and $x^{\circ}\in U_{\epsilon}$. Here, $[x^{*},x^{\circ}]$ is the
line segment connecting $x^{*}$ and $x^{\circ}$. Figure \ref{fig:pf-creep}
may be helpful in understanding the steps of the proof. We fix $(x^{*},y^{*})\in[U_{\epsilon}\times V_{\epsilon}]\cap\gph(S)$
and $x^{\circ}\in U_{\epsilon}$ and continue with the proof.

\begin{figure}
\begin{tabular}{|c|c|}
\hline 
Step 1 & Step 2\tabularnewline
\hline 
\hline 
\includegraphics[scale=0.5]{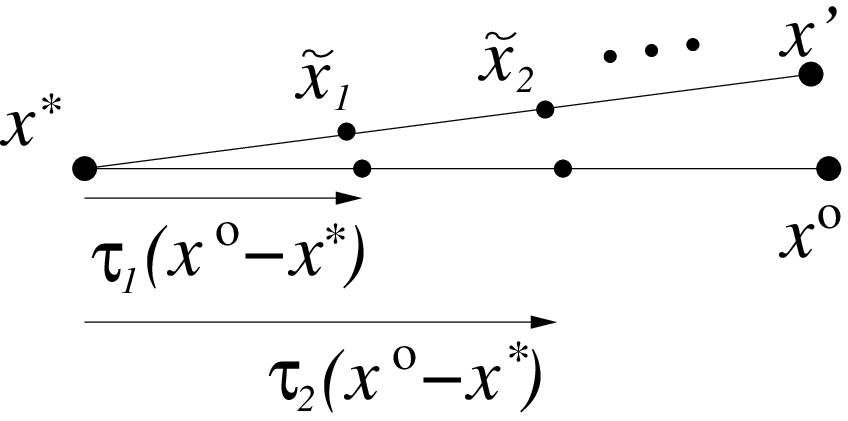} & \includegraphics[scale=0.5]{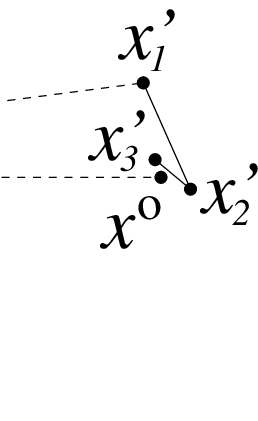}\tabularnewline
\hline 
\end{tabular}

\caption{\label{fig:pf-creep}Steps 1 and 2 of the proof of Lemma \ref{lem:creeping}.}

\end{figure}

\textbf{Step 1: There are $(x^{\prime},y^{\prime})\in\gph(S)$ such
that
\begin{eqnarray*}
 &  & \|x^{\prime}-x^{\circ}\|<(\delta+\epsilon)\|x^{\circ}-x^{*}\|\\
 & \mbox{and } & y^{\prime}-y^{*}\in-(H+\delta+\epsilon)(x^{*}-x^{\circ}).
\end{eqnarray*}
}To simplify notation, let $\tilde{p}:=x^{\circ}-x^{*}$. Let $\bar{\tau}$
be the supremum of all $\tau\in[0,1]$ such that there exists $(x^{\prime},y^{\prime})\in\gph(S)$
satisfying
\begin{eqnarray}
 &  & \|[x^{\prime}-x^{*}]-\tau\tilde{p}\|<(\delta+\epsilon)\tau\|\tilde{p}\|\nonumber \\
 & \mbox{and } & y^{\prime}-y^{*}\in-\tau(H+\delta+\epsilon)(-\tilde{p}).\label{eq:lemma-two-condns}
\end{eqnarray}

Given $(x,y)\in[U\times V]\cap\gph(S)$ and a direction $\tilde{p}\in X\backslash\{0\}$,
there are $p^{\prime}\in X$ such that $\|\tilde{p}-p^{\prime}\|<\delta\|\tilde{p}\|$
and $q^{\prime}\in-(H+\delta)(-\tilde{p})$ such that $(p^{\prime},q^{\prime})\in T_{\scriptsize\gph(S)}(x,y)$.
By the definition of tangent cones, for any $\lambda\in(0,\epsilon)$,
there is some $(x_{\lambda},y_{\lambda})\in\gph(S)$ such that $\|(x_{\lambda},y_{\lambda})-(x,y)\|<\lambda$,
$\|[x_{\lambda}-x]-tp^{\prime}\|\leq\lambda t\|\tilde{p}\|$ and $\|[y_{\lambda}-y]-tq^{\prime}\|\leq\lambda t\|\tilde{p}\|$
for some $t\in(0,1)$. We thus have
\begin{eqnarray*}
 &  & \|(x_{\lambda},y_{\lambda})-(x,y)\|<\lambda,\\
 &  & \|[x_{\lambda}-x]-t\tilde{p}\|<(\delta+\lambda)\|t\tilde{p}\|,\\
 & \mbox{and } & y_{\lambda}-y\in t[-(H+\delta)(-\tilde{p})+\lambda\|\tilde{p}\|\mathbb{B}]\\
 &  & \phantom{y_{\lambda}-y}=-(H+\delta+\lambda)(-t\tilde{p}).
\end{eqnarray*}

Taking $(x,y)=(x^{*},y^{*})$ gives us $\bar{\tau}>0$. Let $(\tilde{x}_{1},\tilde{y}_{1})\in\gph(S)$
and $\tau_{1}\in(0,1]$ be such that \eqref{eq:lemma-two-condns}
holds for $(x^{\prime},y^{\prime})=(\tilde{x}_{1},\tilde{y}_{1})$
and $\tau=\tau_{1}$. If $\tau_{1}<1$, we can use the existence of
some $(p_{1}^{\prime},q_{1}^{\prime})\in T_{\scriptsize{\gph(S)}}(\tilde{x}_{1},\tilde{y}_{1})$
and obtain the existence of $(\tilde{x}_{2},\tilde{y}_{2})\in\gph(S)$
and $\tau_{2}\in(\tau_{1},1]$ such that
\begin{eqnarray*}
 &  & \|[\tilde{x}_{2}-\tilde{x}_{1}]-(\tau_{2}-\tau_{1})\tilde{p}\|<(\delta+\epsilon)(\tau_{2}-\tau_{1})\|\tilde{p}\|\\
 & \mbox{and } & \tilde{y}_{2}-\tilde{y}_{1}\in-(\tau_{2}-\tau_{1})(H+\delta+\epsilon)(-\tilde{p}).
\end{eqnarray*}
The implication
\begin{eqnarray*}
\tilde{y}_{1}-y^{*} & \in & -\tau_{1}(H+\delta+\epsilon)(-\tilde{p})\\
\mbox{and }\tilde{y}_{2}-\tilde{y}_{1} & \in & -(\tau_{2}-\tau_{1})(H+\delta+\epsilon)(-\tilde{p})\\
\mbox{implies }\tilde{y}_{2}-y^{*} & \in & -\tau_{2}(H+\delta+\epsilon)(-\tilde{p}).
\end{eqnarray*}
requires the convexity of $(H+\delta+\epsilon)(-\tilde{p})$. These
conditions imply that \eqref{eq:lemma-two-condns} holds for $(x^{\prime},y^{\prime})=(\tilde{x}_{2},\tilde{y}_{2})$
and $\tau=\tau_{2}$. Similarly, we can obtain a Cauchy sequence $\{(x_{i},y_{i})\}$
with limit $(\tilde{x},\tilde{y})$ and $\tau_{i}\nearrow\tilde{\tau}$
such that \eqref{eq:lemma-two-condns} holds for $(x^{\prime},y^{\prime})=(x_{i},y_{i})$
and $\tau=\tau_{i}$. 

The previous steps showed us that:
\begin{enumerate}
\item [(a)]If \eqref{eq:lemma-two-condns} holds for $(x^{\prime},y^{\prime})=(x_{i},y_{i})$
and $\tau=\tau_{i}$, then we can find $(x_{i+1},y_{i+1})\in\gph(S)$
and $\tau_{i+1}$ such that $\tau_{i}<\tau_{i+1}\leq1$ and \eqref{eq:lemma-two-condns}
holds for $(x^{\prime},y^{\prime})=(x_{i+1},y_{i+1})$ and $\tau=\tau_{i+1}$.
Moreover, if $\tau_{i}<1$ for all $i$, then the sequence $\{(x_{i},y_{i})\}_{i}$
thus obtained is a Cauchy sequence.
\end{enumerate}
Another property that is easy to check is that:
\begin{enumerate}
\item [(b)]If \eqref{eq:lemma-two-condns} holds for $(x^{\prime},y^{\prime})=(x_{i},y_{i})$
and $\tau=\tau_{i}$ for sequences $\{(x_{i},y_{i})\}_{i}\subset\gph(S)$
and $\{\tau_{i}\}_{i}$ such that $\tau_{i}$ is an increasing sequence
with $\tau_{i}\leq1$ for all $i$. Suppose further that $\{(x_{i},y_{i})\}_{i}$
constructed by (a), and let $(\tilde{x},\tilde{y})$ be $\lim_{i\to\infty}(x_{i},y_{i})$
(which lies in $\gph(S)$) and $\tilde{\tau}$ be $\lim_{i\to\infty}\tau_{i}$.
Then \eqref{eq:lemma-two-condns} holds for $(x^{\prime},y^{\prime})=(\tilde{x},\tilde{y})$
and $\tau=\tilde{\tau}$.
\end{enumerate}
By making use of (a) and (b) alternately, we can find $(x^{\prime},y^{\prime})\in\gph(S)$
satisfying \eqref{eq:lemma-two-condns} for $\tau=1$. This proves
the claim in step 1.

\textbf{Step 2: Wrapping up}

So far, we have shown that for all $(x^{*},y^{*})\in[U_{\epsilon}\times V_{\epsilon}]\cap\gph(S)$
and $x^{\circ}\in U_{\epsilon}$, we can find $(x^{\prime},y^{\prime})\in[U\times V]\cap\gph(S)$
such that 
\begin{eqnarray*}
 &  & \|x^{\prime}-x^{\circ}\|<(\delta+\epsilon)\|x^{\circ}-x^{*}\|\\
 & \mbox{and } & y^{*}-y^{\prime}\in(H+\delta+\epsilon)(x^{*}-x^{\circ}).
\end{eqnarray*}

Write $(x_{1}^{\prime},y_{1}^{\prime})=(x^{\prime},y^{\prime})$,
and $\tilde{p}_{1}=x^{\circ}-x_{1}^{\prime}$. Using a similar process
as outlined in step 1 and also the fact that we can find $(p_{1}^{\prime},q_{1}^{\prime})\in T_{\scriptsize\gph(S)}(x_{1}^{\prime},y_{1}^{\prime})$
such that $\|\tilde{p}_{1}-p_{1}^{\prime}\|\leq\delta\|\tilde{p}_{1}\|$
and $q_{1}^{\prime}\in-(H+\delta)(-\tilde{p}_{1})$, we can find $(x_{2}^{\prime},y_{2}^{\prime})\in\gph(S)$
such that
\begin{eqnarray}
 &  & \|x_{2}^{\prime}-x^{\circ}\|<(\delta+\epsilon)\|x^{\circ}-x_{1}^{\prime}\|\nonumber \\
 & \mbox{and } & y_{1}^{\prime}-y_{2}^{\prime}\in(H+\delta+\epsilon)(x_{1}^{\prime}-x^{\circ})\label{eq:start-contain}\\
 &  & \phantom{y_{1}^{\prime}-y_{2}^{\prime}}\subset[\|H\|^{+}+\delta+\epsilon]\|x^{\circ}-x_{1}^{\prime}\|\mathbb{B}.\nonumber 
\end{eqnarray}
Note that $\|x^{\circ}-x_{1}^{\prime}\|<(\epsilon+\delta)\|x^{\circ}-x^{*}\|$.
The condition \eqref{eq:3-hypo-2} was defined so that step 1 can
be applied here to find $(x_{2}^{\prime},y_{2}^{\prime})$. Formula
\eqref{eq:start-contain} implies 
\begin{eqnarray*}
 &  & \|x_{2}^{\prime}-x^{\circ}\|<(\delta+\epsilon)^{2}\|x^{\circ}-x^{*}\|,\\
 & \mbox{and } & y_{1}^{\prime}-y_{2}^{\prime}\subset[\|H\|^{+}+\delta+\epsilon]\|x^{\circ}-x_{1}^{\prime}\|\mathbb{B}\\
 &  & \phantom{y_{1}^{\prime}-y_{2}^{\prime}}\subset(\epsilon+\delta)[\|H\|^{+}+\delta+\epsilon]\|x^{\circ}-x^{*}\|\mathbb{B}.
\end{eqnarray*}
Likewise, we can find $(x_{i}^{\prime},y_{i}^{\prime})\in\gph(S)$
inductively such that
\begin{eqnarray*}
 &  & \|x_{i}^{\prime}-x^{\circ}\|<(\delta+\epsilon)^{i}\|x^{\circ}-x^{*}\|,\\
 & \mbox{and } & y_{i-1}^{\prime}-y_{i}^{\prime}\subset(\delta+\epsilon)^{i-1}[\|H\|^{+}+\delta+\epsilon]\|x^{\circ}-x^{*}\|\mathbb{B}.
\end{eqnarray*}
The sequence $\{(x_{i}^{\prime},y_{i}^{\prime})\}$ is Cauchy, and
hence converges to a limit in the closed set $\gph(S)\cap[U\times V]$.
The $x$ coordinate of this limit is $x^{\circ}$. Let the $y$-coordinate
of this limit be $y^{\circ}$. Since $0<\delta+\epsilon<1$, we have
\begin{eqnarray*}
y^{*}-y^{\circ} & = & y^{*}-y_{1}^{\prime}+\sum_{i=1}^{\infty}[y_{i}^{\prime}-y_{i+1}^{\prime}]\\
 & \in & (H+\delta+\epsilon)(x^{*}-x^{\circ})+[\|H\|^{+}+\delta+\epsilon]\frac{\delta+\epsilon}{1-\delta-\epsilon}\|x^{\circ}-x^{*}\|\mathbb{B}.
\end{eqnarray*}
This gives
\begin{eqnarray*}
y^{*} & \in & y^{\circ}+\left(H+\delta+\epsilon+[\|H\|^{+}+\delta+\epsilon]\frac{\delta+\epsilon}{1-\delta-\epsilon}\right)(x^{*}-x^{\circ})\\
 & \subset & S(x^{\circ})+\left(H+\delta+\epsilon+[\|H\|^{+}+\delta+\epsilon]\frac{\delta+\epsilon}{1-\delta-\epsilon}\right)(x^{*}-x^{\circ}).
\end{eqnarray*}
Since $(x^{*},y^{*})$ is arbitrarily chosen in $[U_{\epsilon}\times V_{\epsilon}]\cap\gph(S)$
and $x^{\circ}$ is arbitrarily chosen in $U_{\epsilon}$, we are
done. 
\end{proof}
We now continue with the proof of Theorem \ref{thm:Aubin-crit}(b).
\begin{proof}
\textbf{{[}Theorem \ref{thm:Aubin-crit}(b){]}} Since $S$ is locally
closed at $(\bar{x},\bar{y})$, we can always reduce the neighborhoods
$U$ and $V$ if necessary so that $[U\times V]\cap\gph(S)$ is closed.
The condition $DS(x|y)(p)\cap[-(H+\delta)(-p)]\neq\emptyset$ easily
implies the existence of $(p^{\prime},q^{\prime})$ satisfying \eqref{eq:Aubin-compare}.
(In fact, the vector $p^{\prime}$ in \eqref{eq:Aubin-compare} can
be chosen to be $p$.) Therefore the conditions in Lemma \ref{lem:creeping}
are satisfied. Since the $\epsilon$ and $\delta$ in the statement
of Lemma \ref{lem:creeping} are arbitrary, we have the pseudo strict
$H$-differentiability of $S$ as needed.\textbf{}
\end{proof}
  The proof of Theorem \ref{thm:Aubin-crit}(c) follows with minor
modifications from the methods in \cite{DQZ06}, which were in turn
motivated by the proof of \cite[Theorem 3.2.4]{Aub91} due to Frankowska.
\begin{proof}
\textbf{{[}Theorem \ref{thm:Aubin-crit}(c){]}} Condition $(2^{\prime})$
is identical to Condition $(2)$ except for the use of the convexified
graphical derivative $D^{\star\star}S(x\mid y)$. We show that Conditions
$(2^{\prime})$ and $(2)$ are equivalent under the added conditions.
It is clear that $(2)\Rightarrow(2^{\prime})$, so we only need to
prove the opposite direction. Our proof is a slight amendment of Step
3 in the proof of \cite[Theorem 1.2]{DQZ06}. 

Suppose Condition $(2^{\prime})$ holds. Fix some $\delta>0$. For
any sets $A,B\subset X\times Y$, denote $d(A,B)$ by $d(A,B):=\inf\{\|a-b\|\mid a\in A,b\in B\}$.
Let us fix $(x,y)\in\gph(S)\cap[U\times V]$ and $p\in X\backslash\{0\}$.
Let $w\in-(H+\delta)(-p)$ and $(p^{*},q^{*})\in\gph\big(DS(x\mid y)\big)$
be such that 
\[
\|(p,w)-(p^{*},q^{*})\|=d\big(\{p\}\times[-(H+\delta)(-p)],\gph\big(DS(x\mid y)\big)\big).
\]
Observe that the point $(p^{*},q^{*})$ is the unique projection of
any point in the open segment $\big((p^{*},q^{*}),(p,w)\big)$ on
$\gph\big(DS(x\mid y)\big)$ under the Euclidean norm. We will prove
that $(p^{*},q^{*})=(p,w)$ and this will prove that $w\in DS(x\mid y)(p)\cap[-(H+\delta)(-p)]$.

By the definition of the graphical derivative, there exists sequences
$t_{n}\searrow0$, $p_{n}\to p^{*}$, $q_{n}\to q^{*}$ such that
$y+t_{n}q_{n}\in S(x+t_{n}p_{n})$ for all $n$. Let $(x_{n},y_{n})$
be a point in $\cl\,\gph(S)$ which is closest to $(x,y)+\frac{t_{n}}{2}(p^{*}+p,q^{*}+w)$
(a projection, not necessarily unique, of the latter point on the
closure of $\gph(S)$). Since $(x,y)\in\gph(S)$ we have
\[
\left\Vert (x,y)+\frac{t_{n}}{2}(p^{*}+p,q^{*}+w)-(x_{n},y_{n})\right\Vert \leq\frac{t_{n}}{2}\|(p^{*}+p,q^{*}+w)\|,
\]
and hence 
\begin{align*}
\|(x,y)-(x_{n},y_{n})\|\leq & \left\Vert (x,y)+\frac{t_{n}}{2}(p^{*}+p,q^{*}+w)-(x_{n},y_{n})\right\Vert +\frac{t_{n}}{2}\|(p^{*}+p,q^{*}+w)\|\\
\leq & t_{n}\|(p^{*}+p,q^{*}+w)\|.
\end{align*}
Thus for $n$ sufficiently large, we have $(x_{n},y_{n})\in U\times V$
and hence $(x_{n},y_{n})\in\gph(S)\cap[U\times V]$. Setting $(\bar{p}_{n},\bar{q}_{n})=(x_{n}-x,y_{n}-y)/t_{n}$,
we deduce by the usual property of a projection (under the Euclidean
norm) that 
\[
\frac{1}{2}(p^{*}+p,q^{*}+w)-(\bar{p}_{n},\bar{q}_{n})\in[T_{\scriptsize\gph(S)}(x_{n},y_{n})]^{0}=\big[\gph\big(D^{\star\star}S(x_{n}\mid y_{n})\big)\big]^{0}.
\]
By the assumptions in (2$^{\prime}$), there exists $w_{n}\in D^{\star\star}S(x_{n}\mid y_{n})(p)\cap[-(H+\delta)(-p)]$
and we have from the above relation 
\begin{equation}
\left\langle \frac{p^{*}+p}{2}-\bar{p}_{n},p\right\rangle +\left\langle \frac{q^{*}+w}{2}-\bar{q}_{n},w_{n}\right\rangle \leq0.\label{eq:DQZ-27}
\end{equation}
We claim that $(\bar{p}_{n},\bar{q}_{n})$ converges to $(p^{*},q^{*})$
as $n\to\infty$. Indeed, 
\begin{align*}
 & \left\Vert \left(\frac{p^{*}+p}{2},\frac{q^{*}+w}{2}\right)-(\bar{p}_{n},\bar{q}_{n})\right\Vert \\
= & \frac{1}{t_{n}}\left\Vert (x,y)+t_{n}\left(\frac{p^{*}+p}{2},\frac{q^{*}+w}{2}\right)-(x_{n},y_{n})\right\Vert \\
\leq & \frac{1}{t_{n}}\left\Vert (x,y)+t_{n}\left(\frac{p^{*}+p}{2},\frac{q^{*}+w}{2}\right)-(x,y)-t_{n}(p_{n},q_{n})\right\Vert \\
= & \left\Vert \left(\frac{p^{*}+p}{2},\frac{q^{*}+w}{2}\right)-(p_{n},q_{n})\right\Vert .
\end{align*}
Therefore, $(\bar{p}_{n},\bar{q}_{n})$ is a bounded sequence and
then, since $y_{n}=y+t_{n}\bar{q}_{n}\in S(x_{n})=S(x+t_{n}\bar{p}_{n})$,
every cluster point $(\bar{p},\bar{q})$ of it belongs to $\gph\big(DS(x\mid y)\big)$.
Moreover, $(\bar{p},\bar{q})$ satisfies
\[
\left\Vert \left(\frac{p^{*}+p}{2},\frac{q^{*}+w}{2}\right)-(\bar{p},\bar{q})\right\Vert \leq\left\Vert \left(\frac{p^{*}+p}{2},\frac{q^{*}+w}{2}\right)-(p^{*},q^{*})\right\Vert .
\]
The above inequality together with the fact that $(p^{*},q^{*})$
is the unique closest point to $\frac{1}{2}(p^{*}+p,q^{*}+w)$ in
$\gph\big(DS(x\mid y)\big)$ implies that $(\bar{p},\bar{q})=(p^{*},q^{*})$.
Our claim is proved. 

Up to a subsequence, $w_{n}$ satisfying \eqref{eq:DQZ-27} converges
to some $\bar{w}\in-(H+\delta)(-p)$. Passing to the limit in \eqref{eq:DQZ-27}
one obtains
\begin{equation}
\left\langle p-p^{*},p\right\rangle +\left\langle w-q^{*},\bar{w}\right\rangle \leq0.\label{eq:DQZ-28}
\end{equation}
Since $(p,w)$ is the unique closest point of $(p^{*},q^{*})$ to
the closed convex set $\{p\}\times[-(H+\delta)(-p)]$, we have
\begin{equation}
\left\langle w-q^{*},w-\bar{w}\right\rangle \leq0.\label{eq:DQZ-29}
\end{equation}
Finally, since $(p^{*},q^{*})$ is the unique closest point to $\frac{1}{2}(p^{*}+p,q^{*}+w)$
in $\gph\big(DS(x\mid y)\big)$ which is a closed cone, we get
\begin{equation}
\left\langle p-p^{*},p^{*}\right\rangle +\left\langle w-q^{*},q^{*}\right\rangle =0.\label{eq:DQA-30}
\end{equation}
In view of \eqref{eq:DQZ-28}, \eqref{eq:DQZ-29} and \eqref{eq:DQA-30},
we obtain
\begin{align*}
 & \|(p,w)-(p^{*},q^{*})\|^{2}\\
= & \left\langle w-q^{*},w-\bar{w}\right\rangle +(\left\langle p-p^{*},p\right\rangle +\left\langle w-q^{*},\bar{w}\right\rangle )-(\left\langle p-p^{*},p^{*}\right\rangle +\left\langle w-q^{*},q^{*}\right\rangle )\leq0.
\end{align*}
Hence $p=p^{*}$ and $w=q^{*}$. We have $DS(x\mid y)(p)\cap[-(H+\delta)(p)]$
containing at least the element $w$, so it cannot be empty. Since
$\delta>0$, $(x,y)\in\gph(S)\cap[U\times V]$ and $p\in X\backslash\{0\}$
are arbitrary, we have Condition (2) in Theorem \ref{thm:Aubin-crit}
as needed.
\end{proof}
As a corollary to Theorem \ref{thm:Aubin-crit}, we obtain the Aubin
criterion as proved in \cite{DQZ06}. 
\begin{thm}
\label{thm:DQZ-Aubin}(Aubin Criterion) Suppose Assumption \ref{ass:common}
holds and $X$ is complete. Let 
\[
\alpha:=\limsup_{(x,y)\xrightarrow[\scriptsize\gph(S)]{}(\bar{x},\bar{y})}\|DS(x\mid y)\|^{-}.
\]

\begin{enumerate}
\item [(a)]We have $\lip\, S(\bar{x}\mid\bar{y})\leq\alpha$, and equality
holds if $Y$ is finite dimensional.
\item [(b)]If both $X$ and $Y$ are finite dimensional Euclidean spaces,
then 
\[
\lip\, S(\bar{x}\mid\bar{y})=\limsup_{(x,y)\xrightarrow[\scriptsize\gph(S)]{}(\bar{x},\bar{y})}\|D^{\star\star}S(x\mid y)\|^{-}.
\]

\end{enumerate}
\end{thm}
\begin{proof}
Recall that $S$ has the Aubin property at $(\bar{x},\bar{y})\in\gph(S)$
if and only if it is pseudo strictly $H$-differentiable there for
$H$ defined by $H(w):=\kappa\|w\|\mathbb{B}$. For a given $(x,y)\in\gph(S)$,
the smallest value of $\kappa\geq0$ such that $DS(x\mid y)(p)\cap\kappa\|p\|\mathbb{B}\neq\emptyset$
for all $p\neq0$ is $\|DS(x\mid y)\|^{-}$. 

We apply these observations to Theorem \ref{thm:Aubin-crit}. In (a),
if $\alpha<\infty$, the condition $\lip\: S(\bar{x}\mid\bar{y})\leq\alpha$
holds by Theorem \ref{thm:Aubin-crit}(b). The statement is trivially
true if $\alpha=\infty$. If $\lip\: S(\bar{x}\mid\bar{y})=\infty$,
then $\lip\: S(\bar{x}\mid\bar{y})\leq\alpha$ from before gives $\alpha=\infty$.
When $Y$ is finite dimensional and $\lip\: S(\bar{x}\mid\bar{y})$
is finite, it follows from Theorem \ref{thm:Aubin-crit}(a) that $\lip\, S(\bar{x}\mid\bar{y})=\alpha$.
For (b), the proof is similar.
\end{proof}
Theorem \ref{thm:DQZ-Aubin}(b) was also proved with viability theory
in \cite{Aub06}. 

We remark on the similarities between Lemma \ref{lem:creeping} and
Aubin's original results. For a set-valued map $S:X\rightrightarrows Y$,
the \emph{inverse} $S^{-1}:Y\rightrightarrows X$ is defined by $S^{-1}(y):=\{x\mid y\in S(x)\}$,
and satisfies $\gph(S^{-1})=\{(y,x)\mid(x,y)\in\gph(S)\}$.
\begin{rem}
(Comparison to Aubin's original results) Let $H:X\rightrightarrows Y$
be defined by $H(w):=\kappa\|w\|\mathbb{B}$. In \cite[Theorem 7.5.4]{AE84}
and \cite[Theorem 5.4.3]{AF90}, the necessary condition in both results
(up to some rephrasing) is that there are neighborhoods $U$ of $\bar{x}$
and $V$ of $\bar{y}$ such that for all $p\in X\backslash\{0\}$
and $(x,y)\in\gph(S)\cap[U\times V]$, there exists $q^{\prime}$
and $w$ such that
\[
p\in[DS(x\mid y)]^{-1}(q^{\prime})+w,\mbox{ }q^{\prime}\in-H(-p)\mbox{ and }\|w\|<\delta\|p\|.
\]
Let $p^{\prime}=p-w$. Then $\|p-p^{\prime}\|<\delta\|p\|$ and $(p^{\prime},q^{\prime})\in T_{\scriptsize\gph(S)}(x,y)$,
so the condition in \eqref{eq:Aubin-compare} is satisfied.
\end{rem}

\section{\label{sec:second}A second characterization: Limits of graphical
derivatives}

While conditions $(2)$ and $(2^{\prime})$ in Theorem \ref{thm:Aubin-crit}
characterize the generalized derivative, the presence of the term
$\delta$ may make these conditions difficult to check in practice.
In Subsection \ref{sub:second-char}, we present another characterization
of the generalized derivatives $H:X\rightrightarrows Y$ that may
be easier to check than Theorem \ref{thm:Aubin-crit}, especially
in the finite dimensional case. More specifically, consider $\{G_{i}\}_{i\in I}$,
where $G_{i}:X\rightrightarrows Y$ are positively homogeneous and
$I$ is some index set. We impose further conditions so that $S$
is pseudo strictly $H$-differentiable at $(\bar{x},\bar{y})$ if
and only if 
\[
G_{i}(p)\cap[-H(-p)]\neq\emptyset\mbox{ for all }i\in I\mbox{ and }p\in X\backslash\{0\}.
\]

\subsection{A generalized inner semicontinuity condition}

Before we move on to the next subsection for a second characterization
of generalized derivatives $H:X\rightrightarrows Y$ such that $S:X\rightrightarrows Y$
is pseudo strictly $H$-differentiable at $(\bar{x},\bar{y})$, we
propose a generalized notion of lower semicontinuity to simplify the
results in Subsection \ref{sub:second-char}. Only Definition \ref{def:gen-isc}
will be important for discussions beyond this subsection, but the
rest of this subsection provides motivation and insights of Definition
\ref{def:gen-isc}.

We say that $S:\mathbb{R}^{n}\rightrightarrows\mathbb{R}^{m}$ is
a \emph{piecewise polyhedral }map if $\gph(S)\subset\mathbb{R}^{n}\times\mathbb{R}^{m}$
is a \emph{piecewise polyhedral} set, i.e., expressible as the union
of finitely many polyhedral sets. The tangent cones at any two points
in the relative interior of a face of a polyhedron are the same, which
leads us to the following result.
\begin{prop}
\label{pro:piecewise-poly-tangent-cones}(Finitely many tangent cones
for piecewise polyhedral maps) Suppose $S:\mathbb{R}^{n}\rightrightarrows\mathbb{R}^{m}$
is piecewise polyhedral. For any point $(\bar{x},\bar{y})\in\gph(S)$,
there is a finite set $\{T_{i}\}_{i\in I}\subset\mathbb{R}^{n}\times\mathbb{R}^{m}$
such that $T_{\scriptsize\gph(S)}(x,y)=T_{i}$ for some $i\in I$
whenever $(x,y)\in\gph(S)$. In particular, if $(x,y)\in\gph(S)$
is close enough to $(\bar{x},\bar{y})$, then $T_{\scriptsize\gph(S)}(x,y)=T_{T_{\scriptsize\gph(S)}}(\tilde{x},\tilde{y})$
for some $(\tilde{x},\tilde{y})$. 
\end{prop}
We next state a piecewise polyhedral example.
\begin{example}
\label{exa:piece-poly}(Piecewise polyhedral $S_{1}:\mathbb{R}\rightrightarrows\mathbb{R}$)
Consider the piecewise polyhedral set-valued map $S_{1}:\mathbb{R}\rightrightarrows\mathbb{R}$
defined by 
\begin{equation}
S_{1}(x):=(-\infty,-|x|]\cup[|x|,\infty).\label{eq:Def-S1}
\end{equation}
See Figure \ref{fig:2-maps} for a diagram of $S_{1}$. The possibilities
for $T_{\scriptsize\gph(S_{1})}(x,y)$, where $(x,y)\in\gph(S_{1})$
and $(x,y)$ is close to $(0,0)$ are $\gph(G_{i})$ for $i\in\{1,\dots,6\}$,
where $G_{i}:\mathbb{R}\rightrightarrows\mathbb{R}$ are defined by
\begin{eqnarray}
G_{1}(x) & = & [x,\infty)\nonumber \\
G_{2}(x) & = & (-\infty,x]\nonumber \\
G_{3}(x) & = & [-x,\infty)\label{eq:T-6}\\
G_{4}(x) & = & (-\infty,-x]\nonumber \\
G_{5}(x) & = & \mathbb{R}\nonumber \\
G_{6}(x) & = & S_{1}(x).\nonumber 
\end{eqnarray}

\end{example}
(In this case, $\gph(S_{1})$ is a cone, so it doesn't matter if $(x,y)$
were not close to $(0,0)$ or not. But in the general case, we would
need $(x,y)$ to be close enough to $(0,0)$.) Notice that for the
map $S_{1}$ defined in \eqref{eq:Def-S1}, while $T_{\scriptsize\gph(S_{1})}:\gph(S_{1})\rightrightarrows\mathbb{R}\times\mathbb{R}$
is not inner semicontinuous at $(0,0)$, the possible limits for $\{T_{\scriptsize\gph(S_{1})}(x_{j},y_{j})\}$,
where $(x_{j},y_{j})\to(0,0)$, take on only a finite number of possibilities
as stated in Proposition \ref{pro:piecewise-poly-tangent-cones}.
We now define a generalized inner semicontinuity that gets around
this difficulty. 
\begin{defn}
\label{def:gen-isc}(Generalized inner semicontinuity) Let $\{T_{i}\}_{i\in I}\subset Y$,
where $I$ is some index set, and let $C\subset X$. For a closed-valued
mapping $S:C\rightrightarrows Y$ and a point $\bar{x}\in C\subset X$,
$S$ is said to be $\{T_{i}\}_{i\in I}$\emph{-inner semicontinuous
(}or $\{T_{i}\}_{i\in I}$\emph{-isc) with respect to $C$ at $\bar{x}$}
if for all $\rho>0$ and $\epsilon>0$, there exists a neighborhood
$V$ of $\bar{x}$ such that for all $x\in C\cap V$, there is some
$i\in I$ such that $T_{i}\cap\rho\mathbb{B}\subset S(x)+\epsilon\mathbb{B}$.
\end{defn}
In the case where $|I|=1$ and $T_{1}=S(\bar{x})$, $\{T_{i}\}_{i\in I}$-inner
semicontinuity reduces to the definition of inner semicontinuity.
Going back to the map $S_{1}$ in \eqref{eq:Def-S1}, we note that
the map $T_{\scriptsize\gph(S_{1})}:\gph(S_{1})\rightrightarrows\mathbb{R}\times\mathbb{R}$
is $\{\gph(G_{i})\}_{i\in I}$-isc at $(0,0)$, where the $G_{i}:\mathbb{R}\rightrightarrows\mathbb{R}$
are defined in \eqref{eq:T-6}. The choice of $\{G_{i}\}_{i\in I}$
is not unique. We can instead define $I=\{1,2\}$ and $G_{i}$ by
\begin{eqnarray}
 &  & G_{1}(x)=x\mbox{ and }G_{2}(x)=-x,\label{eq:T-2}
\end{eqnarray}
and $T_{\scriptsize\gph(S_{1})}:\gph(S_{1})\rightrightarrows\mathbb{R}\times\mathbb{R}$
will still be $\{\gph(G_{i})\}_{i\in I}$-isc at $(0,0)$. Of course,
the $\{\gph(G_{i})\}_{i\in I}$ defined in \eqref{eq:T-2} cannot
be limits of $T_{\scriptsize\gph(S_{1})}(x,y)$ as $(x,y)\xrightarrow[\scriptsize\gph(S_{1})]{}(0,0)$
(in the sense of set convergence in \cite[Definition 4.1]{RW98}).
See Lemma \ref{lem:on-slimsup} for a criterion in finite dimensions.

The index set $I$ need not be finite, as the following examples show.
\begin{example}
\label{exa:Infinite-index-set}(Infinite index set $I$) We give two
examples where the index set $I$ is necessarily infinite if Theorem
\ref{thm:Char-deriv}(a)(b) can be applied.

(a) Consider the function $f_{1}:\mathbb{R}\to\mathbb{R}$ defined
by 
\begin{equation}
f_{1}(x)=\begin{cases}
0 & \mbox{if }x=0\\
x^{2}\sin(1/x) & \mbox{otherwise,}
\end{cases}\label{eq:Def-f-1}
\end{equation}
which has Fr\'{e}chet derivative 
\[
f_{1}^{\prime}(x)=\begin{cases}
0 & \mbox{if }x=0\\
2x\sin(1/x)-\cos(1/x) & \mbox{otherwise.}
\end{cases}
\]
The map $T_{\scriptsize\gph(f_{1})}:\gph(f_{1})\rightrightarrows\mathbb{R}\times\mathbb{R}$
is $\{\gph(G_{\lambda})\}_{\lambda\in[-1,1]}$-isc at $(0,0)$, where
$G_{\lambda}:\mathbb{R}\to\mathbb{R}$ is the linear map with gradient
$\lambda$. 

(b) Next, consider the function $f_{2}:\mathbb{R}^{2}\to\mathbb{R}$
defined by $f_{2}(x)=\|x\|_{2}$. The map $T_{\scriptsize\gph(f_{2})}:\gph(f_{2})\rightrightarrows\mathbb{R}^{2}\times\mathbb{R}$
is $[\{\gph(G_{\lambda})\}_{\|\lambda\|=1}\cup T_{\scriptsize\gph(f_{2})}(0,0)]$-isc
at $(0,0)$, where $G_{\lambda}:\mathbb{R}^{2}\to\mathbb{R}$ is the
linear map with gradient $\lambda\in\mathbb{R}^{2}$. The map $f_{2}$
shows that the index set $I$ in Definition \ref{def:gen-isc} can
be infinite, even when the function is single-valued and semi-algebraic.

\end{example}

\subsection{\label{sub:second-char}A second characterization of generalized
derivatives}

For Theorem \ref{thm:Char-deriv} below, assume that the norm in $X\times Y$
is defined by $\|(p,q)\|_{X\times Y}:=\big\|(\|p\|_{X},\|q\|_{Y})\big\|_{(2)}$,
where $\|\cdot\|_{(2)}$ is the Euclidean norm in $\mathbb{R}^{2}$.
\begin{thm}
\label{thm:Char-deriv}(Characterization of generalized derivative)
Suppose Assumption \ref{ass:common} holds. Let $I$ be some index
set, and $G_{i}:X\rightrightarrows Y$ be positively homogeneous maps
for all $i\in I$. Consider the conditions
\begin{enumerate}
\item $S$ is pseudo strictly $H$-differentiable at $(\bar{x},\bar{y})$.
\item $G_{i}(p)\cap[-H(-p)]\neq\emptyset$ for all $p\in X\backslash\{0\}$
and $i\in I$.
\end{enumerate}
Then the following hold:
\begin{enumerate}
\item [(a)]Suppose $Y$ is finite dimensional and $H$ is compact-valued.
If for all $i\in I$, there exists $\{(x_{j},y_{j})\}\subset\gph(S)$
such that $(x_{j},y_{j})\to(\bar{x},\bar{y})$ and 
\[
\limsup_{j\to\infty}T_{\scriptsize\gph(S)}(x_{j},y_{j})\subset\gph(G_{i}),
\]
 then (1) implies (2). 
\item [(b)]Suppose $\|H\|^{+}$ is finite, $H$ is convex-valued, $X$
is complete, and the mapping $T_{\scriptsize\gph(S)}:\gph(S)\rightrightarrows X\times Y$
is $\{\gph(G_{i})\}_{i\in I}$-isc at $(\bar{x},\bar{y})$. Then (2)
implies (1).
\end{enumerate}
The modified statements hold if the mapping $T_{\scriptsize\gph(S)}:\gph(S)\rightrightarrows X\times Y$
was replaced by the mapping $\cl\,\co\, T_{\scriptsize{\gph}(S)}:\gph(S)\rightrightarrows X\times Y$
to the closed convex hull of the tangent cone instead in (a) and (b).
We now assume that
\begin{equation}
X\mbox{ and }Y\mbox{ are finite dimensional Euclidean spaces.}\label{eq:X-Y-finite}
\end{equation}

\begin{enumerate}
\item [(a$^{\prime}$)]Suppose \eqref{eq:X-Y-finite} holds, and $H$ is
compact-valued. If for all $i\in I$, there exists $\{(x_{j},y_{j})\}\subset\gph(S)$
such that $(x_{j},y_{j})\to(\bar{x},\bar{y})$ and 
\[
\limsup_{j\to\infty}\cl\,\co\, T_{\scriptsize{\gph}(S)}(x_{j},y_{j})\subset\gph(G_{i}),
\]
 then (1) implies (2). 
\item [(b$^{\prime}$)]Suppose \eqref{eq:X-Y-finite} holds, and $H$ is
a prefan. If the mapping $\cl\,\co\, T_{\scriptsize{\gph}(S)}:\gph(S)\rightrightarrows X\times Y$
is $\{\gph(G_{i})\}_{i\in I}$-isc at $(\bar{x},\bar{y})$, then (2)
implies (1).
\end{enumerate}
\end{thm}
\begin{proof}
\textbf{(a) }Suppose $S$ is pseudo strictly $H$-differentiable at
$(\bar{x},\bar{y})$. For each $G_{i}$, there is a sequence $\{(x_{j},y_{j})\}_{j}\subset\gph(S)$
converging to $(\bar{x},\bar{y})$ such that $\limsup_{j\to\infty}T_{\scriptsize\gph(S)}(x_{j},y_{j})\subset\gph(G_{i})$.
Fix some $p\neq0$. By Theorem \ref{thm:Aubin-crit}(a), there is
some $\delta_{j}\searrow0$ such that 
\[
DS(x_{j}\mid y_{j})(p)\cap[-(H+\delta_{j})(-p)]\neq\emptyset.
\]
Let $q_{j}$ be in the LHS of the above, and let $\bar{q}$ be a cluster
point of $\{q_{j}\}_{j=1}^{\infty}$, which exists by the compactness
of $H(-p)$. Since $\gph(G_{i})\supset\limsup_{j\to\infty}T_{\scriptsize\gph(S)}(x_{j},y_{j})$,
we have $(p,\bar{q})\in\gph(G_{i})$, and so $G_{i}(p)\cap[-H(-p)]$
contains $\bar{q}$. Hence $G_{i}(p)\cap[-H(-p)]\neq\emptyset$, which
holds for all $p\neq0$, and we are done.

\textbf{(b) }Given $\gamma>0$, we have neighborhoods $U$ of $\bar{x}$
and $V$ of $\bar{y}$ such that for all $(x,y)\in[U\times V]\cap\gph(S)$,
we have 
\begin{equation}
\gph(G_{i})\cap\mathbb{B}_{X\times Y}\subset[T_{\scriptsize\gph(S)}(x,y)]+\gamma\mathbb{B}_{X\times Y}\label{eq:g-isc-used}
\end{equation}
for some $i\in I$, where $\mathbb{B}_{X\times Y}$ is the unit ball
in $X\times Y$. Let $(x^{*},y^{*})\in\gph(S)\cap[U\times V]$, and
let $i^{*}$ be such that \eqref{eq:g-isc-used} holds for $(x,y)=(x^{*},y^{*})$
and $i=i^{*}$. Choose $p\in X\backslash\{0\}$. Since $G_{i^{*}}(p)\cap[-H(-p)]\neq\emptyset$,
choose $q\in G_{i^{*}}(p)\cap[-H(-p)]$. Then $(p,q)\in\gph(G_{i^{*}})$.
Since $\gph(G_{i^{*}})$ is a cone, we first rescale $(p,q)$ so that
$\|(p,q)\|=1$, even if $\|(p,q)\|<1$. From the fact that $\|H\|^{+}$
is finite, and the equivalence of finite dimensional norms, there
is some $\kappa>0$ such that 
\begin{eqnarray*}
1 & = & \|(p,q)\|\\
 & \leq & \kappa(\|p\|+\|q\|)\\
 & \leq & \kappa(\|p\|+\|H\|^{+}\|p\|)\\
 & = & \kappa\|p\|(1+\|H\|^{+}).
\end{eqnarray*}
 Recall that the choice of $\gamma$ in view of the generalized inner
semicontinuity property gives us some $(p^{\prime},q^{\prime})\in T_{\scriptsize\gph(S)}(x^{*},y^{*})$
such that 
\[
\|(p^{\prime},q^{\prime})-(p,q)\|\leq\gamma.
\]
We have 
\[
\|p-p^{\prime}\|\leq\gamma\leq\kappa\gamma(1+\|H\|^{+})\|p\|\mbox{ and }\|q-q^{\prime}\|\leq\gamma\leq\kappa\gamma(1+\|H\|^{+})\|p\|.
\]
The formula involving $q$ implies that $q^{\prime}\in-[H+\kappa\gamma(1+\|H\|^{+})](-p)$.
If $\gamma$ is chosen so that $\kappa\gamma(1+\|H\|^{+})<\delta$
and $(p,q)$ were rescaled to what they originally were then we can
check that the conditions in Lemma \ref{lem:creeping} are satisfied,
which easily implies that $S$ is pseudo strictly $H$-differentiable
at $(\bar{x},\bar{y})$.

\textbf{(a$^{\prime}$) }The conditions in (a$^{\prime}$) imply that
of (a), which in turn implies the conclusion in (a).

\textbf{(b$^{\prime}$) }The proof for this statement requires added
details from that of (b). Using the methods in the proof of (b), given
$\gamma>0$, we have neighborhoods $U$ of $\bar{x}$ and $V$ of
$\bar{y}$ such that for all $(x^{*},y^{*})\in[U\times V]\cap\gph(S)$
and $p\in X\backslash\{0\}$, there exists $(p^{\prime},q^{\prime})\in\cl\,\co\, T_{\scriptsize{\gph}(S)}(x^{*},y^{*})$
such that 
\[
\|p-p^{\prime}\|\leq\gamma\leq\kappa\gamma(1+\|H\|^{+})\|p\|\mbox{ and }q^{\prime}\in-[H+\kappa\gamma(1+\|H\|^{+})](-p).
\]

The current proof now departs from that in (b). We first claim that
we can reduce $U$ and $V$ if necessary so that $\|D^{\star\star}S(x\mid y)\|^{-}<(1+\|H\|^{+})$
for all $(x,y)\in[U\times V]\cap\gph(S)$. Seeking a proof by contradiction
to the claim, suppose there exists a sequence $\{(x_{j},y_{j})\}_{j}\subset\gph(S)$
such that $(x_{j},y_{j})\to(\bar{x},\bar{y})$ and $\|D^{\star\star}S(x_{j}\mid y_{j})\|^{-}\geq(1+\|H\|^{+})$
for all $j$. Then by \cite[Theorem 4.18]{RW98}, $\{\gph\big(D^{\star\star}S(x_{j}\mid y_{j})\big)\}_{j=1}^{\infty}$
has a subsequence that converges in the set-valued sense to $\gph(\tilde{G})$,
where $\tilde{G}:X\rightrightarrows Y$ is positively homogeneous
and $\|\tilde{G}\|^{-}\geq(1+\|H\|^{+})$. We must then have $\tilde{G}(p)\cap[-H(-p)]=\emptyset$
for some $p\in X\backslash\{0\}$. By the generalized inner semicontinuity
property, there is some $i^{\prime}\in I$ such that $G_{i^{\prime}}\subset\tilde{G}$,
giving us $G_{i^{\prime}}(p)\cap[-H(-p)]=\emptyset$, which is a violation
of the assumption in (2). 

From the claim we just proved, we have $\|D^{\star\star}S(x^{*}\mid y^{*})\|^{-}<1+\|H\|^{+}$.
Note also that $D^{\star\star}S(x^{*}\mid y^{*})$ is graphically
convex and positively homogeneous. By the Aubin criterion in Theorem
\ref{thm:DQZ-Aubin}, $D^{\star\star}S(x^{*}\mid y^{*})$ has the
Aubin property with $\lip\, D^{\star\star}S(x^{*}\mid y^{*})\leq1+\|H\|^{+}$.
We have 
\begin{eqnarray*}
q^{\prime} & \in & D^{\star\star}S(x^{*}\mid y^{*})(p^{\prime})\\
 & \subset & D^{\star\star}S(x^{*}\mid y^{*})(p)+(1+\|H\|^{+})\|p-p^{\prime}\|\mathbb{B}.
\end{eqnarray*}
This means that there exists $q^{\prime\prime}$ such that 
\[
\|q^{\prime\prime}-q^{\prime}\|\leq(1+\|H\|^{+})\|p-p^{\prime}\|\leq\kappa\gamma(1+\|H\|^{+})^{2}\|p\|,
\]
and $q^{\prime\prime}\in D^{\star\star}S(x^{*}\mid y^{*})(p)$. So
$q^{\prime\prime}\in-\big(H+\kappa\gamma[2+3\|H\|^{+}+(\|H\|^{+})^{2}]\big)(-p)$,
which implies 
\[
D^{\star\star}S(x^{*}\mid y^{*})(p)\cap\big[-\big(H+\kappa\gamma[2+3\|H\|^{+}+(\|H\|^{+})^{2}]\big)(-p)\big]\neq\emptyset.
\]
Since $\gamma>0$ can be made arbitrarily small, $(x^{*},y^{*})$
is arbitrary in $[U\times V]\cap\gph(S)$ and $p$ is arbitrary in
$X\backslash\{0\}$, we can apply Theorem \ref{thm:Aubin-crit}(c)
and prove what we need.
\end{proof}
We have the following simplification in the Clarke regular case.
\begin{cor}
\label{cor:regular-fd-case}(Clarke regular, finite dimensional case)
Suppose $S:\mathbb{R}^{n}\rightrightarrows\mathbb{R}^{m}$ is locally
closed and $\gph(S)$ is Clarke regular at $(\bar{x},\bar{y})\in\gph(S)$.
Let $H:\mathbb{R}^{n}\rightrightarrows\mathbb{R}^{m}$ be a prefan.
Then $S$ is pseudo strictly $H$-differentiable at $(\bar{x},\bar{y})$
if and only if 
\[
DS(\bar{x}\mid\bar{y})(p)\cap[-H(-p)]\neq\emptyset\mbox{ for all }p\in\mathbb{R}^{n}\backslash\{0\}.
\]
\end{cor}
\begin{proof}
In finite dimensions, the Clarke regularity of $\gph(S)$ is defined
by the inner semicontinuity of $T_{\scriptsize\gph(S)}:\gph(S)\rightrightarrows\mathbb{R}^{n}\times\mathbb{R}^{m}$.
Apply Theorem \ref{thm:Char-deriv}(a) and (b) for $I=\{1\}$ and
$G_{1}\equiv DS(\bar{x}\mid\bar{y})$.
\end{proof}
An easy consequence of the Clarke regularity of $\gph(S)$ is that
the positively homogeneous map $H:\mathbb{R}^{n}\rightrightarrows\mathbb{R}^{m}$
can be chosen to be single-valued.

Before we present Theorem \ref{thm:FD-gen-derv}, we need to look
at a different view of set-valued maps to analyze the tangent cone
mapping. Denote the family of closed nonempty sets in a finite dimensional
Euclidean space $X$ to be $\cls(X)$. It is known that $\cls(X)$
is a metric space under a hyperspace metric (See \cite[Section 4I]{RW98}).
We can write a set-valued map $S:\mathbb{R}^{n}\rightrightarrows\mathbb{R}^{m}$
as $S:\dom(S)\to\cls(\mathbb{R}^{m})$. We shall use $\slimsup$ to
denote the set of all possible limits (i.e., the outer limit) of $S$
in $\cls(\mathbb{R}^{m})$, that is:
\[
\slimsup_{x\xrightarrow[\scriptsize\dom(S)]{}\bar{x}}S(x):=\;\{C\subset\mathbb{R}^{m}\mid\exists x_{j}\xrightarrow[\scriptsize\dom(S)]{}\bar{x}\mbox{ s.t. }S(x_{j})\to C\}.
\]
As an example on the notation $\slimsup$, \cite[Proposition 4.19]{RW98}
can be rephrased as \begin{subequations} 
\begin{align}
\limsup_{x\to\bar{x}}S(x) & =\cup\{C\mid C\in\slimsup_{x\xrightarrow[\scriptsize\dom(S)]{}\bar{x}}S(x)\},\label{eq:limsup-set}\\
\mbox{and }\liminf_{x\to\bar{x}}S(x) & =\cap\{C\mid C\in\slimsup_{x\xrightarrow[\scriptsize\dom(S)]{}\bar{x}}S(x)\}.\label{eq:liminf-set}
\end{align}
\end{subequations}Here are further results on $\slimsup$. 
\begin{lem}
\label{lem:on-slimsup}(Finite dimensional $\slimsup$) Suppose $S:\mathbb{R}^{n}\rightrightarrows\mathbb{R}^{m}$
is closed-valued and $\{C_{i}\}_{i\in I}\subset\cls(\mathbb{R}^{m})$.
Then
\begin{enumerate}
\item [(a)]For all $i\in I$, the following are equivalent: 

\begin{enumerate}
\item [(i)]There is a sequence $\{x_{k}\}\subset\mathbb{R}^{n}$ such that
$x_{k}\to\bar{x}$ and $\limsup_{k\to\infty}S(x_{k})\subset C_{i}$.
\item [(ii)]There is some $D\in\slimsup_{x\to\bar{x}}S(x)$ such that $D\subset C_{i}$.
\end{enumerate}
\item [(b)]If for all $D\in\slimsup_{x\to\bar{x}}S(x)$, there exists $i\in I$
such that $C_{i}\subset D$, then $S$ is $\{C_{i}\}_{i\in I}$-isc
at $\bar{x}$.
\end{enumerate}
\end{lem}
\begin{proof}
\textbf{(a)} The forward direction follows immediately from the fact
that for $x_{k}\to\bar{x}$, we can find a subsequence if necessary
so that $\lim_{k\to\infty}S(x_{k})$ exists and equals to some $D\in\slimsup_{x\to\bar{x}}S(x)$
by a straightforward application of \cite[Theorem 4.18]{RW98}. The
reverse direction is straightforward.

\textbf{(b) }We prove this by contradiction. Suppose that $S$ is
not $\{C_{i}\}_{i\in I}$-isc at $\bar{x}$. That is, there exists
$\epsilon>0$ and $\rho>0$ and a sequence $\{x_{k}\}_{k=1}^{\infty}$
such that $x_{k}\to\bar{x}$ and 
\[
C_{i}\cap\rho\mathbb{B}\not\subset S(x_{k})+\epsilon\mathbb{B}\mbox{ for all }i\in I.
\]
We may choose a subsequence of $\{x_{k}\}_{k=1}^{\infty}$ if necessary
so that $\lim_{k\to\infty}S(x_{k})$ exists. Fix $i^{*}\in I$. A
straightforward application of \cite[Theorem 4.10(a)]{RW98} shows
that 
\[
C_{i^{*}}\not\subset\liminf_{k\to\infty}S(x_{k})=\lim_{k\to\infty}S(x_{k}).
\]
Since $i^{*}$ is arbitrary and $\lim_{k\to\infty}S(x_{k})\in\slimsup_{x\to\bar{x}}S(x)$,
we have a contradiction, and our proof is complete.
\end{proof}
For the tangent cone mapping $T_{\scriptsize\gph(S)}:\gph(S)\to\cls(\mathbb{R}^{n}\times\mathbb{R}^{m})$,
we have 
\begin{align*}
 & \slimsup_{(x,y)\xrightarrow[\scriptsize\gph(S)]{}(\bar{x},\bar{y})}T_{\scriptsize\gph(S)}(x,y)\\
:= & \;\{C\subset\mathbb{R}^{n}\times\mathbb{R}^{m}\mid\exists(x_{j},y_{j})\xrightarrow[\scriptsize\gph(S)]{}(\bar{x},\bar{y})\mbox{ s.t. }T_{\scriptsize\gph(S)}(x_{j},y_{j})\to C\}.
\end{align*}
We now compare the conditions in Theorem \ref{thm:Char-deriv} with
what we can get from the outer limit $\slimsup$.
\begin{thm}
\label{thm:FD-gen-derv}(Finite dimensional characterization of generalized
derivatives) Let $S:\mathbb{R}^{n}\rightrightarrows\mathbb{R}^{m}$
be such that $S$ is locally closed at $(\bar{x},\bar{y})\in\gph(S)$,
and $H:\mathbb{R}^{n}\rightrightarrows\mathbb{R}^{m}$ be a prefan.
Then $S$ is pseudo strictly $H$-differentiable at $(\bar{x},\bar{y})$
if and only if 
\begin{align}
G(p)\cap[-H(-p)]\neq\emptyset\nonumber \\
\mbox{ whenever } & p\in\mathbb{R}^{n}\backslash\{0\}\mbox{ and }\gph(G)\in\slimsup_{(x,y)\xrightarrow[\scriptsize\gph(S)]{}(\bar{x},\bar{y})}\underbrace{T_{\scriptsize\gph(S)}(x,y)}_{\alpha}.\label{eq:FD-gen-form}
\end{align}

The above continues to hold if the term $\alpha$ in \eqref{eq:FD-gen-form}
is replaced by $\cl\,\co\: T_{\scriptsize{\gph}(S)}$, the closed
convex hull of the tangent cone. \end{thm}
\begin{proof}
Let $\slimsup_{(x,y)\xrightarrow[\scriptsize\gph(S)]{}(\bar{x},\bar{y})}T_{\scriptsize\gph(S)}(x,y)=\{\gph(G_{i})\}_{i\in I}$
and use Theorem \ref{thm:Char-deriv} and Lemma \ref{lem:on-slimsup}.
The case of the closed convex hull is similar.
\end{proof}

\section{Examples}

We illustrate how the results in Subsection \ref{sub:second-char}
can be used to characterize the generalized derivatives $H:X\rightrightarrows Y$
in various cases.
\begin{example}
\label{exa:Char-derv}(Characterizing generalized derivatives) We
apply the results in Subsection \ref{sub:second-char} to characterize
the prefans $H:X\rightrightarrows Y$ that are the generalized derivatives
in several functions defined earlier.
\begin{enumerate}
\item Consider the map $S_{1}:\mathbb{R}\rightrightarrows\mathbb{R}$ defined
by $S_{1}(x)=(-\infty,-|x|]\cup[|x|,\infty)$. See Figure \ref{fig:2-maps}.
Let $I=\{1,\dots,6\}$, define $G_{i}:\mathbb{R}\rightrightarrows\mathbb{R}$
as in \eqref{eq:T-6}, and check that $\{G_{i}\}_{i\in I}$ are such
that $\{\gph(G_{i})\}_{i\in I}$ equals $\underset{(x,y)\xrightarrow[\scriptsize\gph(S_{1})]{}(0,0)}{\slimsup}T_{\scriptsize\gph(S_{1})}(x,y)$.
Hence Theorem \ref{thm:FD-gen-derv} is applicable at $(0,0)$. We
observe that 
\begin{align*}
G_{1}(1)\cap[-H(-1)]\neq\emptyset\mbox{ and }G_{4}(1)\cap[-H(-1)]\neq\emptyset\\
\mbox{ is equivalent to } & [-1,1]\subset H(-1),\\
\mbox{ and }G_{2}(-1)\cap[-H(1)]\neq\emptyset\mbox{ and }G_{3}(-1)\cap[-H(1)]\neq\emptyset\\
\mbox{ is equivalent to } & [-1,1]\subset H(1).
\end{align*}
So $S_{1}$ is pseudo strictly $H$-differentiable at $(0,0)$ if
and only if $[-|p|,|p|]\subset H(p)$ for all $p\in\mathbb{R}$. \\
Note that $G_{5}(x)$ does not set any restriction on $H$. Observe
that we only used $G_{i}$ for $i=1,2,3,4$ in \eqref{eq:T-6}. We
can also apply Theorem \ref{thm:FD-gen-derv} with the fact that $\{\gph(G_{i})\}_{i\in\{1,\dots,5\}}$
equals $\underset{(x,y)\xrightarrow[\scriptsize\gph(S_{1})]{}(0,0)}{\slimsup}\cl\,\co\, T_{\scriptsize\gph(S_{1})}(x,y)$
to see that $G_{6}$ is not needed in characterizing the generalized
derivative $H$. 
\item Consider the map $S_{2}:\mathbb{R}\rightrightarrows\mathbb{R}$ defined
by $S_{2}(x):=\{x\}\cup\{-x\}$. See Figure \ref{fig:2-maps}. Let
$I=\{1,2,3\}$, and define $G_{i}:\mathbb{R}\rightrightarrows\mathbb{R}$
by
\begin{eqnarray*}
G_{1}(x) & = & x\\
G_{2}(x) & = & -x\\
G_{3}(x) & = & S_{2}(x).
\end{eqnarray*}
Then $\{G_{i}\}_{i\in I}$ are such that $\{\gph(G_{i})\}_{i\in I}$
equals $\slimsup_{(x,y)\xrightarrow[\scriptsize\gph(S)]{}(0,0)}T_{\scriptsize\gph(S_{2})}(x,y)$,
and hence Theorem \ref{thm:FD-gen-derv} is applicable at $(0,0)$.
We can easily check that $S_{2}$ is pseudo strictly $H$-differentiable
at $(0,0)$ if and only if $[-1,1]\subset H(p)$ for $p=1,-1$. 
\item Consider the maps $f_{1}:\mathbb{R}\to\mathbb{R}$ and $f_{2}:\mathbb{R}^{2}\to\mathbb{R}$
in Example \ref{exa:Infinite-index-set}. With additional work, we
get $f_{i}$ is pseudo strictly $H$-differentiable at $(0,0)$ if
and only if $\mathbb{B}\subset H(p)$ for all $\|p\|=1$ for both
$i=1,2$.
\end{enumerate}
\end{example}
We remark on the assumption that $H$ is convex-valued in many of
the results in this paper.
\begin{rem}
\label{rem:convex-val}(Convex-valuedness of $H$) Consider the set-valued
maps $S_{1}:\mathbb{R}\rightrightarrows\mathbb{R}$ and $S_{2}:\mathbb{R}\rightrightarrows\mathbb{R}$
as defined in Example \ref{exa:Char-derv}. Also define $S_{3}:\mathbb{R}\rightrightarrows\mathbb{R}$
by 
\[
S_{3}(x)=\begin{cases}
\{x\}\cup\{-x\} & \mbox{ if }x\leq0\\
{}[-x,x] & \mbox{ if }x\geq0.
\end{cases}
\]
Define the map $H^{\prime}:\mathbb{R}\rightrightarrows\mathbb{R}$
by $H^{\prime}\equiv S_{2}$. See Figure \ref{fig:2-maps}. Note that
$H^{\prime}$ is not convex-valued, but satisfies all other requirements
in Theorems \ref{thm:Aubin-crit} and \ref{thm:Char-deriv} for both
$S_{1}$ and $S_{3}$. While $S_{1}$ is pseudo strictly $H^{\prime}$-differentiable
at $(0,0)$, $S_{3}$ is not.
\end{rem}

\begin{figure}
\begin{tabular}{|c|c|c|}
\hline 
\includegraphics[scale=0.4]{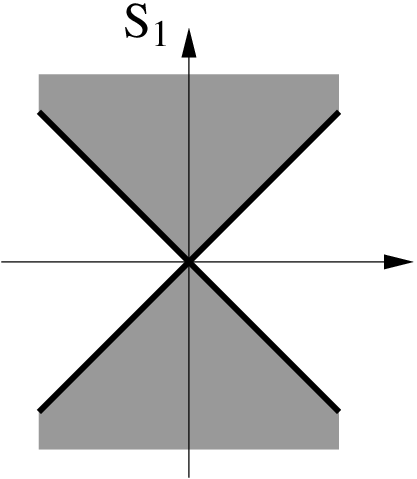} & \includegraphics[scale=0.4]{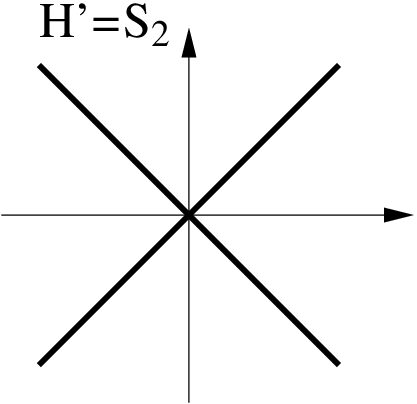} & \includegraphics[scale=0.4]{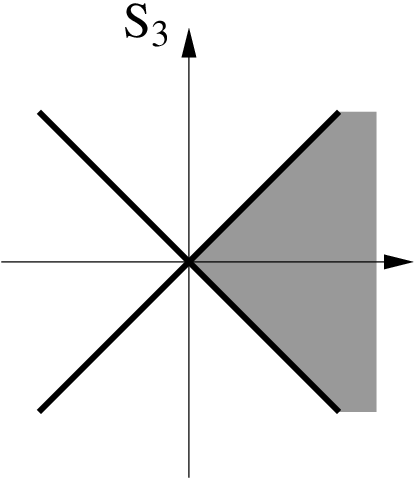}\tabularnewline
\hline 
\end{tabular}

\caption{\label{fig:2-maps}The maps $S_{i}:\mathbb{R}\rightrightarrows\mathbb{R}$
for $i=1,2,3$ are used in Remark \ref{rem:convex-val} and in Examples
\ref{exa:piece-poly} and \ref{exa:Char-derv}.}
\end{figure}

\section{\label{sec:third-char}A third characterization: Extending the normal
cone approach}

For $S:X\rightrightarrows Y$, the Mordukhovich criterion expresses
$\lip\: S(\bar{x}\mid\bar{y})$ in terms of the limiting normal cone.
In this section, we make use of previous results to show how the limiting
normal cone can give a characterization of the generalized derivative
$H:X\rightrightarrows Y$ when $X=\mathbb{R}^{n}$ and $Y=\mathbb{R}^{m}$.
We also show that the convexified coderivatives have a bijective relationship
with the set of possible generalized derivatives. 

We start by defining the limiting normal cone. 
\begin{defn}
\label{def:normal-cones}(Normal cones) For a set $C\subset\mathbb{R}^{n}$,
the \emph{regular normal cone} at $\bar{x}$ is defined as 
\[
\hat{N}_{C}(\bar{x}):=\{y\mid\left\langle y,x-\bar{x}\right\rangle \leq o(\|x-\bar{x}\|)\mbox{ for all }x\in C\}.
\]
The \emph{limiting }(or\emph{ Mordukhovich})\emph{ normal cone} $N_{C}(\bar{x})$
is defined as $\limsup_{x\xrightarrow[C]{}\bar{x}}\hat{N}_{C}(x)$,
or as 
\[
N_{C}(\bar{x})=\{y\mid\mbox{there exists }x_{i}\xrightarrow[C]{}\bar{x},\, y_{i}\in\hat{N}_{C}(x_{i})\mbox{ such that }y_{i}\to y\}.
\]

\end{defn}
In Lemma \ref{lem:normals-come-in} and Theorem \ref{thm:boris-crit}
below, let $\mathbb{R}_{+}=[0,\infty)$ so that for $v\neq0$, $\mathbb{R}_{+}\{v\}$
is the cone generated by $v$. We shall refer to positively homogeneous
maps that have convex graphs as \emph{convex processes}, as is commonly
done in the literature.
\begin{lem}
\label{lem:normals-come-in}(Polar cone criteria) Let $G:\mathbb{R}^{n}\rightrightarrows\mathbb{R}^{m}$
be a convex process with closed graph, and $H:\mathbb{R}^{n}\rightrightarrows\mathbb{R}^{m}$
be a prefan. For each $(u,v)\in[\gph(G)]^{0}\subset\mathbb{R}^{n}\times\mathbb{R}^{m}$,
define $\tilde{G}_{(u,v)}:\mathbb{R}^{n}\rightrightarrows\mathbb{R}^{m}$
by $\gph(\tilde{G}_{(u,v)})=[\mathbb{R}_{+}\{(u,v)\}]^{0}$. Then
\begin{align*}
 & G(p)\cap[-H(-p)]\neq\emptyset\mbox{ for all }p\in\mathbb{R}^{n}\backslash\{0\}\\
\mbox{if and only if }\quad & \tilde{G}_{(u,v)}(p)\cap[-H(-p)]\neq\emptyset\mbox{ for all }p\in\mathbb{R}^{n}\backslash\{0\}\mbox{ and }(u,v)\in[\gph(G)]^{0}.
\end{align*}
\end{lem}
\begin{proof}
It is clear that for each $(u,v)\in[\gph(G)]^{0}$, we have $G\subset\tilde{G}_{(u,v)}$,
so the forward direction is easy. We now prove the reverse direction
by contradiction.

Suppose $G(\bar{p})\cap[-H(-\bar{p})]=\emptyset$ for some $\bar{p}\neq0$.
This means that the convex sets $\gph(G)$ and $\{\bar{p}\}\times[-H(-\bar{p})]$
do not intersect, so there exists some $(\bar{u},\bar{v})\in\mathbb{R}^{n}\times\mathbb{R}^{m}$
and $\alpha\in\mathbb{R}$ such that 
\begin{align*}
\left\langle (\bar{u},\bar{v}),(x,y)\right\rangle <\alpha & \mbox{ for all }(x,y)\in\gph(G)\\
\mbox{and }\left\langle (\bar{u},\bar{v}),(\bar{p},y)\right\rangle >\alpha & \mbox{ for all }y\in[-H(-\bar{p})].
\end{align*}
Since $(0,0)\in\gph(G)$, $\alpha$ must be positive. Furthermore,
since $\gph(G)$ is a cone, we have\begin{subequations} 
\begin{align}
\left\langle (\bar{u},\bar{v}),(x,y)\right\rangle \leq0 & \mbox{ for all }(x,y)\in\gph(G)\label{eq:sep-lem-a}\\
\mbox{and }\left\langle (\bar{u},\bar{v}),(\bar{p},y)\right\rangle >0 & \mbox{ for all }y\in[-H(-\bar{p})].\label{eq:sep-lem-b}
\end{align}
\end{subequations}Note that \eqref{eq:sep-lem-a} implies that $(\bar{u},\bar{v})\in[\gph(G)]^{0}$,
and that $\tilde{G}_{(\bar{u},\bar{v})}:\mathbb{R}^{n}\rightrightarrows\mathbb{R}^{m}$
is defined by $\tilde{G}_{(\bar{u},\bar{v})}(p):=\{y\mid\left\langle (\bar{u},\bar{v}),(p,y)\right\rangle \leq0\}$.
By the definition of $\tilde{G}_{(\bar{u},\bar{v})}$ and \eqref{eq:sep-lem-b},
we have $\tilde{G}_{(\bar{u},\bar{v})}(\bar{p})\cap[-H(-\bar{p})]=\emptyset$,
which is what we need.
\end{proof}
We now recall the definition of coderivatives.
\begin{defn}
(Coderivatives) For a set-valued map $S:\mathbb{R}^{n}\rightrightarrows\mathbb{R}^{m}$
and $(\bar{x},\bar{y})\in\gph(S)$, the \emph{regular coderivative}
at $(\bar{x},\bar{y})$, denoted by $\hat{D}^{*}S(\bar{x}\mid\bar{y}):\mathbb{R}^{m}\rightrightarrows\mathbb{R}^{n}$,
is defined by
\begin{align*}
v\in\hat{D}^{*}S(\bar{x}\mid\bar{y})(u) & \Leftrightarrow(v,-u)\in\hat{N}_{\scriptsize\gph(S)}(\bar{x},\bar{y})\\
 & \Leftrightarrow\left\langle (v,-u),(x,y)-(\bar{x},\bar{y})\right\rangle \leq o\big(\|(x,y)-(\bar{x},\bar{y})\|\big)\\
 & \qquad\qquad\mbox{ for all }(x,y)\in\gph(S).
\end{align*}
The \emph{limiting coderivative }(or \emph{Mordukhovich coderivative})
at $(\bar{x},\bar{y})\in\gph(S)$ is denoted by $D^{*}S(\bar{x}\mid\bar{y}):\mathbb{R}^{m}\rightrightarrows\mathbb{R}^{n}$
and is defined by  
\[
v\in D^{*}S(\bar{x}\mid\bar{y})(u)\Leftrightarrow(v,-u)\in N_{\scriptsize\gph(S)}(\bar{x},\bar{y}).
\]

\end{defn}
In the definitions of both the regular and limiting coderivatives,
the minus sign before $u$ is necessary so that if $f:\mathbb{R}^{n}\to\mathbb{\mathbb{R}}^{m}$
is $\mathcal{C}^{1}$ at $\bar{x}$, then 
\[
D^{*}f(\bar{x}\mid f(\bar{x}))(y)=\nabla f(\bar{x})^{*}y\mbox{ for all }y\in\mathbb{R}^{m}.
\]

We can now state the main result of this section.
\begin{thm}
\label{thm:boris-crit}(Generalized Mordukhovich criterion) Let $S:\mathbb{R}^{n}\rightrightarrows\mathbb{R}^{m}$
be locally closed at $(\bar{x},\bar{y})\in\gph(S)$ and let $H:\mathbb{R}^{n}\rightrightarrows\mathbb{R}^{m}$
be a prefan. Then $S$ is pseudo strictly $H$-differentiable at $(\bar{x},\bar{y})$
if and only if any of the following equivalent conditions hold: 
\begin{enumerate}
\item [(a)]For all $p\in\mathbb{R}^{n}\backslash\{0\}$ and $(v,-u)\in N_{\scriptsize\gph(S)}(\bar{x},\bar{y})$,
there exists $y\in H(p)$ s.t. $\left\langle u,y\right\rangle \leq\left\langle v,p\right\rangle $. 
\item [(b)]For all $p\in\mathbb{R}^{n}\backslash\{0\}$ and $u\in\mathbb{R}^{m}$,
$\underset{y\in H(p)}{\min}\left\langle u,y\right\rangle \leq\underset{v\in D^{*}S(\bar{x}\mid\bar{y})(u)}{\min}\left\langle v,p\right\rangle .$
\item [(c)]For all $p\in\mathbb{R}^{n}\backslash\{0\}$ and $u\in\mathbb{R}^{m}$,
$\underset{y\in H(p)}{\min}\left\langle u,y\right\rangle \leq\underset{v\in\scriptsize{\cl\,\co\,}D^{*}S(\bar{x}\mid\bar{y})(u)}{\min}\left\langle v,p\right\rangle .$
\end{enumerate}
\end{thm}
\begin{proof}
It is clear that (a) is equivalent to (b), and that (b) is equivalent
to (c) by elementary properties of convexity. To simplify notation,
let $\{\gph(G_{i})\}_{i\in I}=\slimsup_{(x,y)\xrightarrow[\scriptsize\gph(S)]{}(\bar{x},\bar{y})}\cl\,\co\: T_{\scriptsize\gph(S)}(x,y)$.
We also use the definition of $\tilde{G}_{(u,v)}:\mathbb{R}^{n}\rightrightarrows\mathbb{R}^{m}$
made in the statement of Lemma \ref{lem:normals-come-in}, and the
following equivalent formulation of (a):
\begin{enumerate}
\item [(a$^{\prime}$)]For all $p\in\mathbb{R}^{n}\backslash\{0\}$ and
$(u,v)\in N_{\scriptsize\gph(S)}(\bar{x},\bar{y})$, there exists
$y\in H(p)$ s.t. $\left\langle (u,v),(p,y)\right\rangle \geq0$. 
\end{enumerate}
Using \cite[Corollary 11.35(b)]{RW98} (which states that a sequence
of closed convex cones converges if and only if the corresponding
sequence of polar cones converges), we have 
\begin{align*}
\slimsup_{(x,y)\xrightarrow[\scriptsize\gph(S)]{}(\bar{x},\bar{y})}\hat{N}_{\scriptsize\gph(S)}(x,y) & =\slimsup_{(x,y)\xrightarrow[\scriptsize\gph(S)]{}(\bar{x},\bar{y})}[\cl\,\co\, T_{\scriptsize\gph(S)}(x,y)]^{0}\\
 & =\{[\gph(G_{i})]^{0}\}_{i\in I}
\end{align*}
By the observation in \eqref{eq:limsup-set}, which recalls \cite[Proposition 4.19]{RW98},
we have 
\begin{align*}
N_{\scriptsize\gph(S)}(\bar{x},\bar{y}) & =\limsup_{(x,y)\xrightarrow[\scriptsize\gph(S)]{}(\bar{x},\bar{y})}\hat{N}_{\scriptsize\gph(S)}(x,y)=\bigcup_{i\in I}[\gph(G_{i})]^{0}.
\end{align*}
By Lemma \ref{lem:normals-come-in}, 
\begin{align*}
 & G(p)\cap[-H(-p)]\neq\emptyset\mbox{ for all }p\in\mathbb{R}^{n}\backslash\{0\}\\
\mbox{if and only if } & \tilde{G}_{(u,v)}(p)\cap[-H(-p)]\neq\emptyset\mbox{ for all }p\in\mathbb{R}^{n}\backslash\{0\}\mbox{ and }(u,v)\in[\gph(G)]^{0}.
\end{align*}
By Theorem \ref{thm:FD-gen-derv}, $S$ is pseudo strictly $H$-differentiable
at $(\bar{x},\bar{y})$ if and only if 
\begin{align*}
 & G_{i}(p)\cap[-H(-p)]\neq\emptyset\mbox{ for all }p\in\mathbb{R}^{n}\backslash\{0\}\mbox{ and }i\in I,\\
\mbox{or equivalently, } & \tilde{G}_{(u,v)}(p)\cap[-H(-p)]\neq\emptyset\mbox{ for all }p\in\mathbb{R}^{n}\backslash\{0\}\mbox{ and }(u,v)\in\bigcup_{i\in I}[\gph(G_{i})]^{0}.
\end{align*}
We can substitute $\bigcup_{i\in I}[\gph(G_{i})]^{0}=N_{\scriptsize\gph(S)}(\bar{x},\bar{y})$
in the above formula. Unrolling the definition of $\tilde{G}_{(u,v)}$
gives: For all $p\in\mathbb{R}^{n}\backslash\{0\}$ and $(u,v)\in N_{\scriptsize\gph(S)}(\bar{x},\bar{y})$,
there exists some $y\in[-H(-p)]$ such that $\left\langle (u,v),(p,y)\right\rangle \leq0$,
which is easily seen to be condition (a$^{\prime}$). 
\end{proof}
Characterizing the generalized derivatives $H:\mathbb{R}^{n}\rightrightarrows\mathbb{R}^{m}$
in terms $N_{\scriptsize\gph(S)}(\bar{x},\bar{y})$ or $D^{*}S(\bar{x}\mid\bar{y})$
instead of the tangent cones not only enjoys a simpler statement,
it also enables one to use tools for normal cones that may not be
present for tangent cones. For example, estimates of the coderivatives
of the composition of two set-valued maps are more easily available
than corresponding results in terms of tangent cones. 

In the particular case of the Aubin property, we obtain the classical
Mordukhovich criterion.
\begin{cor}
\label{rem:classical-boris}(Mordukhovich criterion) Suppose $S:\mathbb{R}^{n}\rightrightarrows\mathbb{R}^{m}$
is osc, and $\bar{y}\in S(\bar{x})$. Then 
\[
\lip\: S(\bar{x}\mid\bar{y})=\|D^{*}S(\bar{x}\mid\bar{y})\|^{+}=\|\cl\,\co\, D^{*}S(\bar{x}\mid\bar{y})\|^{+}.
\]
\end{cor}
\begin{proof}
By Theorem \ref{thm:boris-crit}, $\lip\, S(\bar{x}\mid\bar{y})$
is the infimum of all $\kappa$ such that 
\[
\underset{y:\|y\|\leq\kappa\|p\|}{\min}\left\langle u,y\right\rangle \leq\underset{v\in D^{*}S(\bar{x}\mid\bar{y})(u)}{\min}\left\langle v,p\right\rangle \mbox{ for all }p\in\mathbb{R}^{n}\backslash\{0\}\mbox{ and }u\in\mathbb{R}^{m}.
\]
Now, 
\[
\underset{y:\|y\|\leq\kappa\|p\|}{\min}\left\langle u,y\right\rangle =-\kappa\|u\|\|p\|,
\]
so $\lip\, S(\bar{x}\mid\bar{y})$ is the infimum of all $\kappa$
such that 
\begin{align*}
 & \underset{v\in D^{*}S(\bar{x}\mid\bar{y})(u)}{\max}-\left\langle v,p\right\rangle \leq\kappa\|u\|\|p\|\mbox{ for all }p\in\mathbb{R}^{n}\backslash\{0\}\mbox{ and }u\in\mathbb{R}^{m},\\
\mbox{or } & \underset{v\in D^{*}S(\bar{x}\mid\bar{y})(u)}{\max}\|v\|\|p\|\leq\kappa\|u\|\|p\|\mbox{ for all }p\in\mathbb{R}^{n}\backslash\{0\}\mbox{ and }u\in\mathbb{R}^{m},\\
\mbox{or } & \underset{v\in D^{*}S(\bar{x}\mid\bar{y})(u)}{\max}\|v\|\leq\kappa\|u\|\mbox{ for all }u\in\mathbb{R}^{m}.
\end{align*}
The fact that $\lip\: S(\bar{x}\mid\bar{y})=\|D^{*}S(\bar{x}\mid\bar{y})\|^{+}$
follows easily. The other equality is similar.  
\end{proof}
Theorem \ref{thm:boris-crit}(c) shows that $\cl\,\co\: D^{*}S(\bar{x}\mid\bar{y})$
characterizes all possible generalized derivatives $H:\mathbb{R}^{n}\rightrightarrows\mathbb{R}^{m}$.
As Theorem \ref{thm:co-coderv-best} shows, the reverse holds as well.
\begin{lem}
(Outer semicontinuity of convexified maps)\label{lem:osc-co} Suppose
$D:\mathbb{R}^{m}\rightrightarrows\mathbb{R}^{n}$ is osc, and is
locally bounded at $\bar{x}$. Then the map $\co\: D:\mathbb{R}^{m}\rightrightarrows\mathbb{R}^{n}$,
which maps $x$ to the convex hull of $D(x)$, is osc at $\bar{x}$.\end{lem}
\begin{proof}
It suffices to show that if $y_{i}\in\co\: D(x_{i})$, $y_{i}\to\bar{y}$
and $x_{i}\to\bar{x}$, then $\bar{y}\in\co\: D(\bar{x})$. By Caratheodory's
theorem, we can write $y_{i}$ as a convex combination of $z_{i,1}$,
$z_{i,2}$, ..., $z_{i,n+1}$. By taking a subsequence if necessary,
we can assume that $z_{i,1}$ converges to some $\bar{z}_{1}$ in
$D(\bar{x})$. Doing this $n+1$ times allows us to assume that for
any $j\in\{1,\dots,n+1\}$, $\{z_{i,j}\}_{i=1}^{\infty}$ converges
to some $\bar{z}_{j}\in D(\bar{x})$. It is elementary that $\bar{y}$
is in the convex hull of $\{\bar{z}_{1},\dots,\bar{z}_{n+1}\}$, which
gives $\bar{y}\in\co\: D(\bar{x})$ as needed.
\end{proof}
For $D:\mathbb{R}^{m}\rightrightarrows\mathbb{R}^{n}$ such that $D$
is positively homogeneous and $\|D\|^{+}$ is finite, define $\mathcal{H}(D)$
by
\begin{eqnarray}
\mathcal{H}(D) & := & \{H:\mathbb{R}^{n}\rightrightarrows\mathbb{R}^{m}:H\mbox{ is a prefan,}\nonumber \\
 &  & \qquad\mbox{and for all }p\in\mathbb{R}^{n}\backslash\{0\}\mbox{ and }u\in\mathbb{R}^{m},\label{eq:def-script-H}\\
 &  & \qquad\underset{y\in H(p)}{\min}\left\langle u,y\right\rangle \leq\underset{v\in\scriptsize{\cl\,\co\,}D(u)}{\min}\left\langle v,p\right\rangle \}.\nonumber 
\end{eqnarray}
Suppose $S:\mathbb{R}^{n}\rightrightarrows\mathbb{R}^{m}$ is locally
closed at $(\bar{x},\bar{y})$. By Theorem \ref{thm:boris-crit},
$\mathcal{H}(D^{*}S(\bar{x}\mid\bar{y}))$ is the set of all possible
$H$ with the relevant properties such that $S$ is pseudo strictly
$H$-differentiable at $(\bar{x},\bar{y})$. We now state another
lemma. 
\begin{lem}
\label{lem:conv-coderv-gen-derv}(Strict reverse inclusion property
of $\mathcal{H}(\cdot)$) Suppose $D_{i}:\mathbb{R}^{m}\rightrightarrows\mathbb{R}^{n}$
such that $D_{i}$ is positively homogeneous, osc, and $\|D_{i}\|^{+}$
is finite for $i=1,2$. Then the following hold.
\begin{enumerate}
\item $\cl\,\co\, D_{1}\subset\cl\,\co\, D_{2}$ implies $\mathcal{H}(D_{1})\supset\mathcal{H}(D_{2})$.
\item $\cl\,\co\, D_{1}\neq\cl\,\co\, D_{2}$ implies $\mathcal{H}(D_{1})\neq\mathcal{H}(D_{2})$.
\item $\cl\,\co\, D_{1}\subsetneq\cl\,\co\, D_{2}$ implies $\mathcal{H}(D_{1})\supsetneq\mathcal{H}(D_{2})$.
\item $\mathcal{H}(D_{1})=\mathcal{H}(D_{2})$ implies $\cl\,\co\, D_{1}=\cl\,\co\, D_{2}$.
\end{enumerate}
\end{lem}
\begin{proof}
Property (1) follows easily from the definitions, property (4) is
equivalent to property (2), and property (3) follows easily from property
(1) and (2). We thus concentrate on proving property (4). We shall
assume throughout that $D_{1}$ and $D_{2}$ are convex-valued to
cut down on notation. 

Assume $\mathcal{H}(D_{1})=\mathcal{H}(D_{2})$. We first prove that
$\|D_{1}\|^{+}=\|D_{2}\|^{+}$. Let $\lambda$ be such that the set-valued
map $H_{\lambda}:\mathbb{R}^{n}\rightrightarrows\mathbb{R}^{m}$ defined
by $H_{\lambda}(p):=\lambda\|p\|\mathbb{B}$ lies in $\mathcal{H}(D_{1})$.
We have 
\begin{align*}
\min_{y\in H_{\lambda}(p)}\left\langle u,y\right\rangle  & \leq\min_{v\in D(u)}\left\langle v,p\right\rangle \mbox{ for all }u\in\mathbb{R}^{n}\mbox{ and }p\in\mathbb{R}^{n}\backslash\{0\}\\
\iff-\lambda\|p\|\|u\| & \leq\min_{v\in D(u)}\left\langle v,p\right\rangle \mbox{ for all }u\in\mathbb{R}^{n}\mbox{ and }p\in\mathbb{R}^{n}\backslash\{0\}\\
\iff\max_{v\in D(u)}\left\langle v,-p\right\rangle  & \leq\lambda\|p\|\|u\|\mbox{ for all }u\in\mathbb{R}^{n}\mbox{ and }p\in\mathbb{R}^{n}\backslash\{0\}\\
\iff\max_{v\in D(u)}\|v\| & \leq\lambda\|u\|\mbox{ for all }u\in\mathbb{R}^{n}.
\end{align*}
The above implies that $\|D_{1}\|^{+}=\inf\{\lambda\mid H_{\lambda}\in\mathcal{H}(D_{1})\}$.
Since $\mathcal{H}(D_{1})=\mathcal{H}(D_{2})$, we have $\|D_{1}\|^{+}=\|D_{2}\|^{+}$
as needed. 

Suppose on the contrary that $D_{1}\not\equiv D_{2}$. There must
be some $\bar{u}$ and $\bar{v}$ such that without loss of generality,
$\bar{v}\notin D_{1}(\bar{u})$ but $\bar{v}\in D_{2}(\bar{u})$.
Since $D_{1}(\bar{u})$ is convex, there is some $\bar{w}\neq0$ and
$\alpha\in\mathbb{R}$ such that 
\begin{align*}
 & \left\langle \bar{w},\bar{v}\right\rangle <\alpha\\
\mbox{and } & \left\langle \bar{w},v\right\rangle >\alpha\mbox{ for all }v\in D_{1}(\bar{u}).
\end{align*}
By the outer semicontinuity of $D_{1}$, there is a neighborhood $\mathbb{B}_{\epsilon}(\bar{u})$
of $\bar{u}$ such that $D_{1}(u)\subset\{v\mid\left\langle \bar{w},v\right\rangle >\alpha\}$
for all $u\in\mathbb{B}_{\epsilon}(\bar{u})$. We can suppose $0<\epsilon<2$,
and let the variable $\bar{\theta}>0$ be such that $2\sin(\bar{\theta}/2)=\epsilon$.

We can assume that $\|\bar{w}\|=\|\bar{u}\|=1$. Define $H:\mathbb{R}^{n}\rightrightarrows\mathbb{R}^{m}$
to be 
\begin{align*}
H(x) & :=\begin{cases}
\left\{ y\left|\begin{array}{l}
\frac{1}{2}[\alpha+\left\langle \bar{v},\bar{w}\right\rangle ]\leq\left\langle \bar{u},y\right\rangle \leq\|D_{1}\|^{+},\\
\|y\|^{2}-\left\langle \bar{u},y\right\rangle ^{2}\leq L^{2}
\end{array}\right.\right\}  & \mbox{if }x=\bar{w}\\
\lambda H(\bar{w}) & \mbox{if }x=\lambda\bar{w}\mbox{ for some }\lambda>0\\
\|D_{1}\|^{+}\|x\|\mathbb{B} & \mbox{otherwise},
\end{cases}
\end{align*}
where $L\geq\frac{1}{\sin\bar{\theta}}(\|D_{1}\|^{+}+\frac{1}{2}[\alpha+\left\langle \bar{v},\bar{w}\right\rangle ]\cos\bar{\theta})$.
See Figure \ref{fig:co-coderv-pf} for an illustration of $H(\bar{w})$
when $\alpha>0$.

We now show that $H$ is convex-valued and compact-valued. It only
suffices to show that $H(\bar{w})$ is convex and compact. The set
\[
\left\{ y\mid\frac{1}{2}[\alpha+\left\langle \bar{v},\bar{w}\right\rangle ]\leq\left\langle \bar{u},y\right\rangle \leq\|D_{1}\|^{+}\right\} 
\]
is the intersection of two half spaces and is thus convex. The map
$y\mapsto\|y\|^{2}-\left\langle \bar{u},y\right\rangle ^{2}$ is a
convex quadratic function, so $\{y\mid\|y\|^{2}-\left\langle \bar{u},y\right\rangle ^{2}\leq L^{2}\}$
is convex. To check compactness, we note that 
\[
\|y\|^{2}\leq L^{2}+\left\langle \bar{u},y\right\rangle ^{2}\leq L^{2}+\max\left(\frac{1}{2}[\alpha+\left\langle \bar{v},\bar{w}\right\rangle ],\|D_{1}\|^{+}\right)^{2},
\]
so $H(\bar{w})$ is compact. In addition, it is clear that $H$ is
positively homogeneous and $\|H\|^{+}$ is finite. 

Once the claim below is proved, we will establish the result at hand.

\textbf{Claim:} The map $H$ satisfies $H\in\mathcal{H}(D_{1})$,
but $H\notin\mathcal{H}(D_{2})$. 

Suppose on the contrary $H\in\mathcal{H}(D_{2})$. Since $\bar{v}\in D_{2}(\bar{u})$,
we can find a $y\in H(\bar{w})$ such that $\left\langle \bar{u},y\right\rangle \leq\left\langle \bar{v},\bar{w}\right\rangle $.
But this is not the case since for all $y\in H(\bar{w})$, we have
$\left\langle \bar{u},y\right\rangle \geq\frac{1}{2}[\alpha+\left\langle \bar{v},\bar{w}\right\rangle ]>\left\langle \bar{v},\bar{w}\right\rangle $.

Next, we show that $H\in\mathcal{H}(D_{1})$. We need to check that
for all $p\in\mathbb{R}^{n}\backslash\{0\}$ and $(u,v)\in\gph(D_{1})$,
we can find a $y\in H(p)$ such that $\left\langle u,y\right\rangle \leq\left\langle v,p\right\rangle $.
Since $D_{1}$ is positively homogeneous and $\|D_{1}\|^{+}$ is finite,
we can check only $(u,v)\in\gph(D_{1})$ such that $\|u\|=1$. There
is no need to check for the case $u=0$ and $v\neq0$ because in this
case $\|D_{1}\|^{+}=\infty$. We can further assume that $\|p\|=1$. 

We see that $p\neq\bar{w}$ poses no problems because
\begin{eqnarray*}
\min_{y\in H(p)}\left\langle u,y\right\rangle  & = & -\|u\|\|p\|\|D_{1}\|^{+}\\
 & \leq & -\|p\|\|v\|\\
 & \leq & \left\langle v,p\right\rangle .
\end{eqnarray*}

Let $\partial\mathbb{B}:=\{u\in\mathbb{R}^{m}\mid\|u\|=1\}$. By our
earlier discussion on the outer semicontinuity of $D_{1}$ and the
fact that $\|D_{1}\|^{+}$ is finite, we have
\begin{align*}
[\partial\mathbb{B}\times\mathbb{R}^{n}]\cap\gph(D_{1})\subset & \big([\partial\mathbb{B}\cap\mathbb{B}_{2\sin(\bar{\theta}/2)}(\bar{u})]\times[\|D_{1}\|^{+}\mathbb{B}\cap\{v\mid\left\langle \bar{w},v\right\rangle \geq\alpha\}]\big)\\
 & \cup\big([\partial\mathbb{B}\backslash\mathbb{B}_{2\sin(\bar{\theta}/2)}(\bar{u})]\times\|D_{1}\|^{+}\mathbb{B}\big).
\end{align*}
The possibilities for $(u,v)\in[\partial\mathbb{B}\times\mathbb{R}^{n}]\cap\gph(D_{1})$
are covered in the next two cases.

\begin{figure}
\includegraphics[scale=0.5]{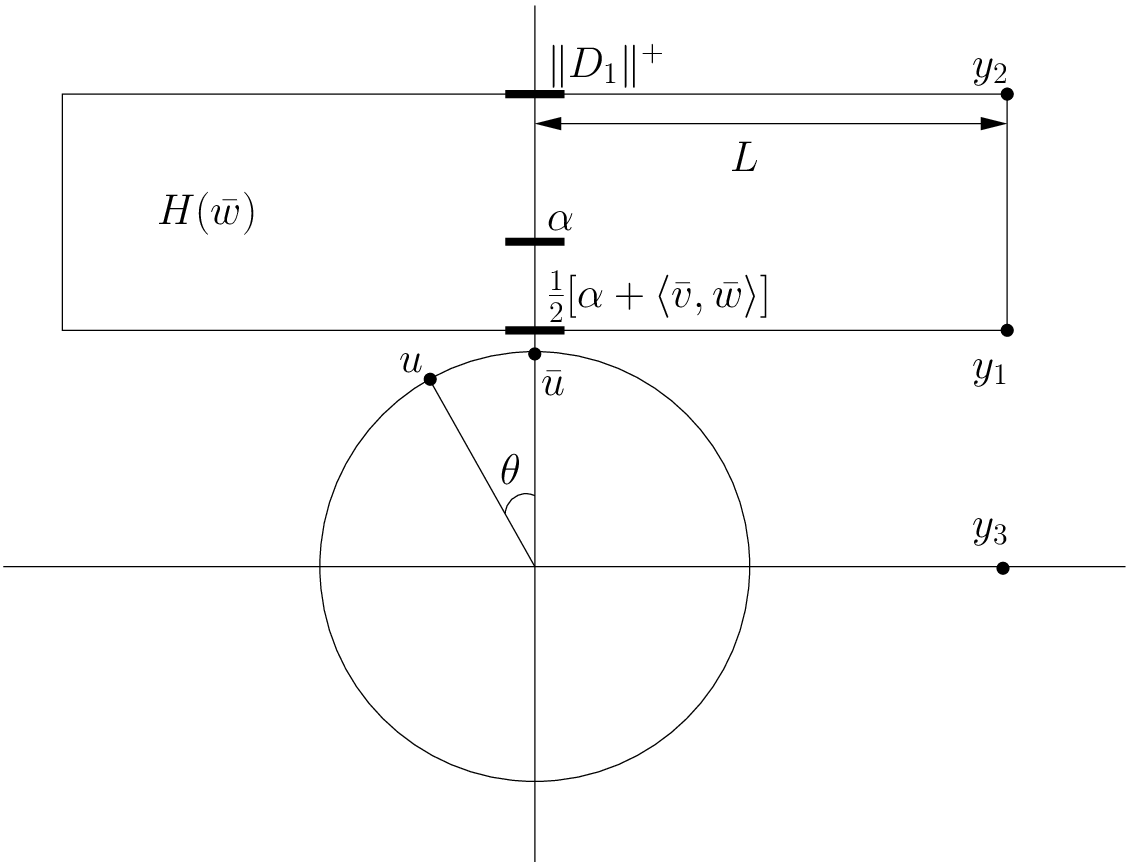}

\caption{\label{fig:co-coderv-pf}In the proof of Lemma \ref{lem:conv-coderv-gen-derv},
we prove that $H:\mathbb{R}^{n}\rightrightarrows\mathbb{R}^{m}$ satisfies
$H\in\mathcal{H}(D_{1})$ for the map $D_{1}:\mathbb{R}^{m}\rightrightarrows\mathbb{R}^{n}$.
The figure shows the distinct features in the two dimensional subspace
in $\mathbb{R}^{m}$ containing $0$, $u$ and $\bar{u}$. The points
$y_{1}$, $y_{2}$ and $y_{3}$ as marked will be used in the proof
of case 2 of Lemma \ref{lem:conv-coderv-gen-derv}.}

\end{figure}

\textbf{Case 1}: If $(u,v)\in[\partial\mathbb{B}\cap\mathbb{B}_{2\sin(\bar{\theta}/2)}(\bar{u})]\times[\|D_{1}\|^{+}\mathbb{B}\cap\{v\mid\left\langle \bar{w},v\right\rangle \geq\alpha\}]$,
then we can find $y\in H(\bar{w})$ s.t. $\left\langle u,y\right\rangle \leq\left\langle v,\bar{w}\right\rangle $.

This case gives $u\in[\partial\mathbb{B}\cap\mathbb{B}_{2\sin(\bar{\theta}/2)}(\bar{u})]$.
Figure \ref{fig:co-coderv-pf} shows the two dimensional subspace
in $\mathbb{R}^{n}$ containing the points $0$, $u$ and $\bar{u}$
in the case when $\alpha>0$. (If $u=\bar{u}$, just take any subspace
passing through $0$ and $\bar{u}$.) The intersection of $H(\bar{w})$
with the subspace is also illustrated. The condition that $u\in\partial\mathbb{B}\cap\mathbb{B}_{2\sin(\bar{\theta}/2)}(\bar{u})$
implies that the angle $\theta$ in Figure \ref{fig:co-coderv-pf}
is in the interval $[0,\bar{\theta}]$. The point $y_{1}$ is formally
defined as the point lying in the two dimensional subspace spanned
by $u$ and $\bar{u}$, and satisfies $\left\langle \bar{u},y_{1}\right\rangle =\frac{1}{2}[\alpha+\left\langle \bar{v},\bar{w}\right\rangle ]$
and $\|y_{1}\|^{2}-\left\langle \bar{u},y_{1}\right\rangle ^{2}=L^{2}$.
By restricting the maximum angle $\bar{\theta}$ if necessary when
$\alpha<0$ and using elementary geometry, we have 
\begin{eqnarray*}
\left\langle u,y_{1}\right\rangle  & \leq & \left\langle \bar{u},y_{1}\right\rangle \\
 & = & \frac{1}{2}[\alpha+\left\langle \bar{v},\bar{w}\right\rangle ]\\
 & < & \alpha\\
 & \leq & \left\langle v,\bar{w}\right\rangle ,
\end{eqnarray*}
which gives us what we need. 

\textbf{Case 2}: If $(u,v)\in[\partial\mathbb{B}\backslash\mathbb{B}_{2\sin(\bar{\theta}/2)}(\bar{u})]\times\|D_{1}\|^{+}\mathbb{B}$,
then we can find $y\in H(\bar{w})$ s.t. $\left\langle u,y\right\rangle \leq\left\langle v,\bar{w}\right\rangle $.

In this case, we need to show that for all $(u,v)$ given, we can
find $y\in H(\bar{w})$ such that $\left\langle u,y\right\rangle \leq-\|D_{1}\|^{+}$.
The fact that $-\|D_{1}\|^{+}\leq\left\langle v,\bar{w}\right\rangle $
(which comes from $\|v\|\leq\|D_{1}\|^{+}\|u\|=\|D_{1}\|^{+}$) will
give us what we need. Once again, see Figure \ref{fig:co-coderv-pf}.
We split this case into two subcases.

\textbf{Subcase 2a: $[\alpha+\left\langle \bar{v},\bar{w}\right\rangle ]\geq0$. }

For the choice of $y_{1}$, we have 
\[
\left\langle u,y_{1}\right\rangle =-L\sin\theta+\frac{1}{2}[\alpha+\left\langle \bar{v},\bar{w}\right\rangle ]\cos\theta.
\]
We first consider $\theta\in[\bar{\theta},\pi/2]$. The RHS of the
above attains its maximum when $\theta=\bar{\theta}$. The condition
$L\geq\frac{1}{\sin\bar{\theta}}(\|D_{1}\|^{+}+\frac{1}{2}[\alpha+\left\langle \bar{v},\bar{w}\right\rangle ]\cos\bar{\theta})$
implies that $\left\langle u,y_{1}\right\rangle \leq-\|D_{1}\|^{+}$
as needed.

We now treat the case where $\theta\in[\pi/2,\pi]$. The point $y_{2}$
is defined similarly as in $y_{1}$, except that $\left\langle \bar{u},y_{2}\right\rangle =\|D_{1}\|^{+}$.
We have
\[
\left\langle u,y_{2}\right\rangle =-L\sin\theta+\|D_{1}\|^{+}\cos\theta.
\]
From Figure \ref{fig:co-coderv-pf}, we can see that the RHS of the
above attains its maximum when $\theta=\pi$, which gives $\left\langle u,y_{2}\right\rangle \leq-\|D_{1}\|^{+}$
as needed.

\textbf{Subcase 2b: $[\alpha+\left\langle \bar{v},\bar{w}\right\rangle ]<0$.}

Repeat the arguments for when $\theta\in[\bar{\theta},\pi/2]$, but
replace all occurrences of $y_{1}$ by $y_{3}$ as marked in Figure
\ref{fig:co-coderv-pf}. (The point $y_{3}$ satisfying $\left\langle \bar{u},y_{3}\right\rangle =0$
will lie in $H(\bar{w})$.) The case when $\theta\in[\pi/2,\pi]$
is also similar. This concludes the proof of the claim, and establishes
(4).
\end{proof}
With the above lemma, we state a theorem on the relationship between
convexified coderivatives and the generalized derivatives.
\begin{thm}
\label{thm:co-coderv-best}(Convexified coderivatives from generalized
derivatives) Suppose $S:\mathbb{R}^{n}\rightrightarrows\mathbb{R}^{m}$
is locally closed at $(\bar{x},\bar{y})\in\gph(S)$ and has the Aubin
property there. Then the convexified coderivative $\cl\,\co\: D^{*}S(\bar{x}\mid\bar{y}):\mathbb{R}^{m}\rightrightarrows\mathbb{R}^{n}$
is uniquely determined by the set of all prefans $H:\mathbb{R}^{n}\rightrightarrows\mathbb{R}^{m}$
such that $S$ is pseudo strictly $H$-differentiable at $(\bar{x},\bar{y})$.\end{thm}
\begin{proof}
The map $D^{*}S(\bar{x}\mid\bar{y})$ is osc, and by Lemma \ref{lem:osc-co},
so is $\cl\,\co\, D^{*}S(\bar{x}\mid\bar{y})$. Apply Lemma \ref{lem:conv-coderv-gen-derv}(4)
to get the result.
\end{proof}

\section{\label{sec:Appl}Applications}

We end this paper by discussing how our results can be applied to
study constraint mappings, to study generalized pseudo strict $H$-differentiability,
metric regularity and linear openness, and to estimate the convexified
limiting coderivative of a limit of set-valued maps.  

In Proposition \ref{pro:RW-E9.44} below, we study constraint mappings,
and shall only treat the case where $D$ is Clarke regular and apply
Corollary \ref{cor:regular-fd-case} to illustrate the spirit of our
results. While stronger conditions for the case where $D$ is not
Clarke regular can be deduced from the characterizations in Sections
\ref{sec:second} and \ref{sec:third-char}, the extra calculations
do not give additional insight.
\begin{prop}
\label{pro:RW-E9.44}(Constraint mappings, adapted from \cite[Example 9.44]{RW98})
Let $S(x)=F(x)-D$ for smooth $F:\mathbb{R}^{n}\to\mathbb{R}^{m}$
 and a closed set $D\subset\mathbb{R}^{m}$ that is Clarke regular
at every point. Suppose also that $(\bar{x},\bar{u})\in\gph(S)$.
Then $S^{-1}(u)$ consists of all $x$ satisfying the constraint system
$F(x)-u\in D$, with $u$ as a parameter.

Suppose $H:\mathbb{R}^{m}\rightrightarrows\mathbb{R}^{n}$ is a prefan
such that for all $p\in\mathbb{R}^{m}\backslash\{0\}$, there exists
$q\in-H(-p)$ such that $\nabla F(\bar{x})q-p\in T_{D}(F(\bar{x})-\bar{u})$.
Then $S^{-1}$ is pseudo strictly $H$-differentiable at $(\bar{u},\bar{x})$.\end{prop}
\begin{proof}
The set $\gph(S)$ is specified by $F_{0}(x,u)\in D$ with $F_{0}(x,u)=F(x)-u$,
i.e., 
\[
(x,u)\in\gph(S)\mbox{ if and only if }F_{0}(x,u)\in D.
\]
The mapping $F_{0}:\mathbb{R}^{n}\times\mathbb{R}^{m}\to\mathbb{R}^{m}$
is smooth, and its Jacobian $\nabla F_{0}(\bar{x},\bar{u})=[\nabla F(\bar{x}),-I]$
has full rank $m$. Applying the rule in \cite[Exercise 6.7]{RW98},
we see that
\begin{eqnarray*}
\hat{N}_{\scriptsize\gph(S)}(\bar{x},\bar{u}) & = & \{(v,-y)\mid y\in\hat{N}_{D}(F(\bar{x})-\bar{u}),v=\nabla F(\bar{x})^{T}y\}\\
 & = & \{(v,-y)\mid y\in N_{D}(F(\bar{x})-\bar{u}),v=\nabla F(\bar{x})^{T}y\}\\
 & = & N_{\scriptsize\gph(S)}(\bar{x},\bar{u}).
\end{eqnarray*}
Therefore, $\gph(S)$ is Clarke regular at $(\bar{x},\bar{u})$. From
\cite[Exercise 6.7]{RW98} again, we see that 
\begin{align*}
T_{\scriptsize\gph(S)}(\bar{x},\bar{u}) & =\{(q,p)\in\mathbb{R}^{n}\times\mathbb{R}^{m}\mid\nabla F(\bar{x})q-p\in T_{D}(F(\bar{x})-\bar{u})\},\\
\mbox{so }T_{\scriptsize\gph(S^{-1})}(\bar{u},\bar{x}) & =\{(p,q)\in\mathbb{R}^{m}\times\mathbb{R}^{n}\mid\nabla F(\bar{x})q-p\in T_{D}(F(\bar{x})-\bar{u})\}.
\end{align*}
The formula for $T_{\scriptsize\gph(S^{-1})}$, together with Corollary
\ref{cor:regular-fd-case}, gives the conclusion needed.
\end{proof}
If the constraint qualification 
\begin{equation}
y\in N_{D}(F(\bar{x})-\bar{u}),\nabla F(\bar{x})^{T}y=0\mbox{ implies }y=0\label{eq:CQ}
\end{equation}
holds in Proposition \ref{pro:RW-E9.44}, then \cite[Exercise 9.44]{RW98}
states that $S^{-1}$ has the Aubin property with modulus 
\[
\max_{{y\in N_{D}(F(\bar{x})-\bar{u})\atop \|y\|=1}}\frac{1}{\|\nabla F(\bar{x})^{T}y\|},
\]
so an $H:\mathbb{R}^{m}\rightrightarrows\mathbb{R}^{n}$ satisfying
the stated conditions can be found. 

The case where $D=\{0\}^{r}\times\mathbb{R}_{-}^{m-r}$ in Proposition
\ref{pro:RW-E9.44} gives 
\[
S^{-1}(u):=\{x:F_{i}(x)=u_{i}\mbox{ for }i=1,\dots,r\mbox{ and }F_{i}(x)\leq u_{i}\mbox{ for }i=r+1,\dots,m\}.
\]
In this case, the constraint qualification \eqref{eq:CQ} is equivalent
to the Mangasarian-Fromovitz constraint qualification defined by the
existence of $w\in\mathbb{R}^{n}$ satisfying 
\begin{eqnarray*}
\nabla F_{i}(\bar{x})w & < & 0\mbox{ for all }i\in\{r+1,\dots,m\}\mbox{ s.t. }F_{i}(\bar{x})=0,\\
\mbox{ and }\nabla F_{i}(\bar{x})w & = & 0\mbox{ for all }i\in\{1,\dots,r\}.
\end{eqnarray*}
The corresponding conclusion in Proposition \ref{pro:RW-E9.44} can
be easily deduced.

Next, we remark that the Aubin property of the constraint mapping
$S^{-1}:\mathbb{R}^{m}\rightrightarrows\mathbb{R}^{n}$ at $(\bar{u},\bar{x})$
is also equivalently studied as the metric regularity of $S:\mathbb{R}^{n}\rightrightarrows\mathbb{R}^{m}$
at $(\bar{x},\bar{u})$. One may refer to standard references \cite{KK02,Mor06,RW98}
for more on metric regularity and its relationship with the Aubin
property. The equivalence between pseudo strict $H$-differentiability
and generalized metric regularity is discussed in \cite[Section 7]{set_diff}. 

Finally, we discuss how Lemma \ref{lem:conv-coderv-gen-derv} can
be used to find the convexified limiting coderivative of a certain
limit of set-valued maps $S_{i}:\mathbb{R}^{n}\rightrightarrows\mathbb{R}^{m}$.
The following result arose in \cite{diff_inc} from trying to calculate
the coderivative of the reachable map in differential inclusions,
where the reachable map can be approximated from a sequence of discretized
reachable maps. This result is of independent interest in the study
of set-valued maps. 
\begin{thm}
\cite{diff_inc}(Convexified limiting coderivative of limits of set-valued
maps) Let $S:\mathbb{R}^{n}\rightrightarrows\mathbb{R}^{m}$ be a
closed set-valued map. Suppose $\{S_{i}(\cdot)\}_{i=1}^{\infty}$,
where $S_{i}:\mathbb{R}^{n}\rightrightarrows\mathbb{R}^{m}$, are
osc set-valued maps such that for any $\epsilon>0$ and $x\in\mathbb{R}^{n}$,
there is some $I$ such that 
\begin{equation}
\mathbf{d}(S(x),S_{i}(x))<\epsilon\mbox{ for all }i>I,\label{eq:Hausdorff-bdd}
\end{equation}
where $\mathbf{d}(\cdot,\cdot)$ denotes the Pompieu-Hausdorff distance
between two closed compact sets. Then we have 
\[
\cl\,\co D^{*}S(\bar{x}\mid\bar{y})\subset\bigcap_{{\delta>0\atop N\in\mathbb{N}}}\cl\,\co\bigcup_{i>N}\bigcup_{{x\in\mathbb{B}_{\delta}(\bar{x})\atop y\in\mathbb{B}_{\delta}(\bar{y})\cap S_{i}(x)}}D^{*}S_{i}(x\mid y).
\]

\end{thm}
The above result is proved by making use of Lemma \ref{lem:conv-coderv-gen-derv}
and showing that for all $\delta>0$ and $N\in\mathbb{N}$, we have
$\mathcal{H}(D^{*}S(\bar{x}\mid\bar{y}))\supset\mathcal{H}\left(\bigcup_{i>N}\bigcup_{{x\in\mathbb{B}_{\delta}(\bar{x})\atop y\in\mathbb{B}_{\delta}(\bar{y})\in S_{i}(x)}}D^{*}S_{i}(x\mid y)\right)$.
We refer to \cite{diff_inc} for more details.

The convexified limiting coderivative is less precise than the limiting
coderivative. But for the problem of estimating the \emph{Clarke subdifferential
}$\cl\,\co\partial f(x)$ of $f$ at $x$, where the \emph{marginal
function }$f$ is defined by $f(x):=\min_{y\in S(x)}\varphi(x,y)$,
it turns out that using $\cl\,\co D^{*}S(x\mid y)$ to estimate $\cl\,\co\partial f(x)$
is not any less precise than using $D^{*}S(x\mid y)$. Once again,
we refer to \cite{diff_inc} for more details.

\section{Acknowledgements}

I thank Alexander Ioffe, Adrian Lewis, Dmitriy Drusvyatskiy and ShanShan
Zhang for conversations that prompted the addition of Section \ref{sec:Appl},
and also to Alexander Ioffe for probing how the results here can be
stated in terms of fans, which simplified some of the statements in
this paper. I thank the anonymous referees for their comments and
suggestions which have helped improve the paper, especially the referee
who read the paper very carefully and pointed out many errors in the
previous version.

\bibliographystyle{amsplain}
\bibliography{../refs}

\end{document}